\documentclass[journal,11pt,onecolumn,draftcls]{IEEEtran}

\usepackage{hyperref}
\usepackage{color}
\usepackage{comment}
\usepackage{cite}
\usepackage{amsmath,amsfonts,amssymb,amsthm,mathrsfs,mathtools,stmaryrd,units}
\usepackage{bm} %lettre grecques en gras ds equations
\usepackage{enumitem}
\usepackage{cleveref}
\usepackage{float}
\usepackage{graphicx}
\usepackage[table,xcdraw]{xcolor}
\usepackage{readarray}
\usepackage{caption}
\usepackage{subcaption}
\usepackage{multirow}
\usepackage[ruled,vlined]{algorithm2e}
\usepackage{enumerate}
\usepackage{lipsum} 

\newtheorem{theorem}{Theorem}
\newtheorem{remark}{Remark}
\newtheorem{lemma}{Lemma}

\newtheorem{corollary}{Corollary}
\usepackage{tikz}
\usepackage{pgfplots}
\usepackage{pgf}
\usetikzlibrary{calc}
\usepgfplotslibrary{groupplots,dateplot}
\newcommand\inputpgf[2]{{
\let\pgfimageWithoutPath\pgfimage
\renewcommand{\pgfimage}[2][]{\pgfimageWithoutPath[##1]{#1/##2}}
\input{#1/#2}
}}
\usetikzlibrary{external}
\DeclareMathOperator{\tr}{Tr}

\pgfplotsset{compat=newest}
\usetikzlibrary{plotmarks}
\usetikzlibrary{arrows.meta}
\usepgfplotslibrary{patchplots}
\usepackage{grffile}
\newlength\fwidth

\newlength\height
\newlength\width

\DeclareMathOperator{\Diff}{D}

\DeclareMathOperator*{\Id}{Id}

\newcommand*{\realSpace}{\ensuremath{\mathbb{R}}}
\newcommand*{\complexSpace}{\ensuremath{\mathbb{C}}}

\newcommand*{\nfeatures}{\ensuremath{p}}

\newcommand*{\dimension}{\ensuremath{d}}

\newcommand*{\manifold}{\ensuremath{\mathcal{M}}}
\newcommand*{\tangentSpace}[1]{\ensuremath{T_{#1}\mathcal{M}}}

\newcommand*{\point}{\ensuremath{\boldsymbol{\theta}}}
\newcommand*{\pointAlt}{\ensuremath{\boldsymbol{\hat{\theta}}}}

\newcommand*{\tangent}{\ensuremath{\boldsymbol{X}_{\point}}}
\newcommand*{\tangentBis}{\ensuremath{\boldsymbol{Y}_{\point}}}
\newcommand*{\tangentTer}{\ensuremath{\boldsymbol{Z}_{\point}}}

\newcommand*{\tangentAlt}{\ensuremath{\boldsymbol{X}_{\pointAlt}}}
\newcommand*{\tangentAltBis}{\ensuremath{\boldsymbol{Y}_{\pointAlt}}}

\newcommand*{\metric}[3]{\ensuremath{\langle#2,#3\rangle_{#1}}}
\newcommand*{\LC}[2]{\ensuremath{\nabla_{#2}#1}}

\newcommand*{\basisElement}{\ensuremath{\boldsymbol{\Omega}}}

\newcommand*{\data}{\ensuremath{\mathbf{y}}}

\newcommand*{\estimator}{\ensuremath{\boldsymbol{\hat{\theta}}}}

\newcommand*{\errVec}{\ensuremath{\boldsymbol{\varepsilon}_{\point}}}
\newcommand*{\errCov}{\ensuremath{\boldsymbol{C}_{\point}}}

\newcommand*{\scoreVec}{\ensuremath{\boldsymbol{s}_{\point}}}

\newcommand*{\errCovBayes}{\ensuremath{\boldsymbol{C}}}

\newcommand*{\FisherMat}{\ensuremath{\boldsymbol{F}}}

\newcommand*{\eye}{\ensuremath{\boldsymbol{I}}}

\usepackage{xcolor}

\usepackage{pgfplotsthemetol}
\definecolor{TolDarkPurple}{HTML}{332288}
\definecolor{TolDarkBlue}{HTML}{6699CC}
\definecolor{TolLightBlue}{HTML}{88CCEE}
\definecolor{TolLightGreen}{HTML}{44AA99}
\definecolor{TolDarkGreen}{HTML}{117733}
\definecolor{TolDarkBrown}{HTML}{999933}
\definecolor{TolLightBrown}{HTML}{DDCC77}
\definecolor{TolDarkRed}{HTML}{661100}
\definecolor{TolLightRed}{HTML}{CC6677}
\definecolor{TolLightPink}{HTML}{AA4466}
\definecolor{TolDarkPink}{HTML}{882255}
\definecolor{TolLightPurple}{HTML}{AA4499}

\definecolor{mDarkBrown}{HTML}{604c38}
\definecolor{mDarkTeal}{HTML}{23373b}
\definecolor{mLightBrown}{HTML}{EB811B}
\definecolor{mLightGreen}{HTML}{14B03D}

\colorlet{mLightTeal}{mDarkTeal!75}
\colorlet{mAltBrown}{mLightBrown!80!mDarkBrown!90!red}
\colorlet{mLightBlue}{TolDarkBlue!70!mDarkTeal}
\colorlet{mLightYellow}{TolLightBrown!70!mLightBrown}
\colorlet{mDarkOrange}{mLightBrown!40!red!75!mDarkTeal}
\definecolor{myblue}{HTML}{1f77b4}
\definecolor{myorange}{HTML}{ff7f0e}
\definecolor{myred}{HTML}{d62728}

\colorlet{mCiteColor}{mLightBlue}

\newtheorem{assumption}{Assumption}

\begin{document}

%\begin{frontmatter}

\title{Intrinsic Bayesian Cramér-Rao Bound with an Application to Covariance Matrix Estimation}
\author{Florent Bouchard, Alexandre Renaux, Guillaume Ginolhac, Arnaud Breloy
\thanks{F. Bouchard and A. Renaux are with CNRS, L2S, CentraleSupélec, Université Paris-Saclay.
G. Ginolhac is with LISTIC (EA3703), University Savoie Mont Blanc.
A. Breloy is with CEDRIC, CNAM.}}

%\author{Antoine Collas, Florent Bouchard, Arnaud Breloy, Guillaume Ginolhac,  Chengfang Ren, Jean-Philippe Ovarlez
%\thanks{
%A. Collas, C. Ren are with SONDRA, CentraleSupélec, Université Paris-Saclay.
%F. Bouchard is with CNRS, L2S, CentraleSupélec, Université Paris-Saclay.
%G. Ginolhac is with LISTIC (EA3703), University Savoie Mont Blanc.
%A. Breloy is with LEME (EA4416), University Paris Nanterre.
%J-P Ovarlez is with SONDRA, CentraleSupélec, Université Paris-Saclay and DEMR, ONERA.
%Part of this work was supported by ANR-ASTRID MARGARITA (ANR-17-ASTR-0015).}
%}

\maketitle

\begin{abstract}
This paper presents a new performance bound for estimation problems where the parameter to estimate lies in a Riemannian manifold (a smooth manifold endowed with a Riemannian metric) and follows a given prior distribution.
In this setup, the chosen Riemannian metric induces a geometry for the parameter manifold, as well as an intrinsic notion of the estimation error measure.
Performance bounds for such error measure were previously obtained in the non-Bayesian case (when the unknown parameter is assumed to deterministic), and referred to as \textit{intrinsic} Cramér-Rao bound.
The presented result then appears either as: \textit{a}) an extension of the intrinsic Cramér-Rao bound to the Bayesian estimation framework; \textit{b}) a generalization of the Van-Trees inequality (Bayesian Cramér-Rao bound) that accounts for the aforementioned geometric structures.
In a second part, we leverage this formalism to study the problem of covariance matrix estimation when the data follow a Gaussian distribution, and whose covariance matrix is drawn from an inverse Wishart distribution. 
Performance bounds for this problem are obtained for both the mean squared error (Euclidean metric) and the natural Riemannian distance for Hermitian positive definite matrices (affine invariant metric).
Numerical simulation illustrate that assessing the error with the affine invariant metric is revealing of interesting properties of the maximum a posteriori and minimum mean square error estimator, which are not observed when using the Euclidean metric. 
\end{abstract}


%\end{frontmatter}

%% INTRO
\section{Introduction}
\label{sec:intro}

Performance bounds are fundamental quantities that allow to characterize the optimal accuracy (generally assessed in terms of mean squared error) that can be achieved for a fixed setup of a given estimation problem.
Such bounds provide useful tools that can be used to evaluate the validity of estimation procedures, or to design systems so that a certain level of accuracy is actually achievable in practice.
When the unknown parameter to be estimated is assumed to be deterministic, the most famous of these is the Cramér-Rao bound \cite{Cramer1946, Rao1992}, and many works addressed its generalization and refinements, e.g., to predict so-called threshold phenomena (sometimes observed at low sample support and/or low SNR when the observation model is non-linear) in estimation performance \cite{Bhattacharyya1946, Barankin1949, Todros2010, Chaumette2008, Routtenberg2014}.
In Bayesian estimation, the unknown parameter is rather assumed to follow a known distribution, that reflects some prior knowledge.
A well-known performance bound for this context is provided by the Van Trees inequality \cite{vantrees2001}, also referred to as the Bayesian Cramér-Rao bound.
As for the deterministic case, many variations, generalizations, and refinements were proposed in the literature \cite{vantrees2007, Weinstein88, Todros2010, Chaumette2017}.

In practice, the parameter to be estimated can often satisfy a set of constraints dictated by the model assumptions (e.g., a fixed normalization). 
If these constraints can be expressed in a system of equations, the so-called constrained Cramér-Rao bound \cite{Gorman1990, stoica1998cramer, marzetta1993simple, ben2009constrained} generalize the standard (i.e. non-Bayesian) Cramér-Rao bound to account for this inherent structure.
It was further studied in \cite{al2018constrained, nitzan2018cram, nitzan2019cramer} (e.g., to account for estimation bias), and extended to other classes of bounds in \cite{ren2015constrained, nitzan2023barankin}.
The extension to the Bayesian context was less studied in the literature. On can cite \cite{andersson2017constrained} which deals with discrete-time nonlinear filtering, or \cite{prevost2020cramer} where random constraints were addressed.
Unfortunately, in many cases, the set of constraints cannot easily be expressed in a system of equations, but actually yields a smooth manifold (e.g., linear subspaces, and covariance matrices).
It becomes then possible to leverage tools from Riemannian geometry, which offers several advantages to derive new performance bounds \cite{Hendriks1991, smith05, Xavier2005, barrau2013note, boumal2013intrinsic, boumal2014cramer, boumal2014thesis, Bonnabel2015, Breloy2019, Solo2020, Labsir2022, Hendriks2023, bouchard2023fisher}.
Notably, this paper will focus on the formalism of intrinsic (non-Bayesian) Cramér-Rao bounds proposed in \cite{smith05} (see also introductions in \cite{barrau2013note, boumal2014thesis, bouchard2023fisher}).
This work proposed a generalization of the (non-Bayesian) Cramér-Rao inequality that is obtained when the error vector is defined from the Riemannian logarithm.
Depending on the choice of the Riemannian metric, this results allows for bounding a geometric distance rather than the usual mean squared error.
As expected, the intrinsic Cramér-Rao bound coincides with the standard Cramér-Rao bound in the Euclidean case (i.e., the parameter lies in $\mathbb{R}^d$ endowed with the Euclidean metric).
However, in the non-Euclidean geometry, it can reveal unexpected properties that are inherent to the estimation problem and its geometry (intrinsic bias, effect of the curvature of the manifold, etc.). 
These properties can provide a better assessment of the estimation performance actually experienced in practice, which motivates further development of this framework.

In this work, we propose to extend some of the results of \cite{smith05} to a Bayesian estimation context.
The contributions of this paper are the following:
\begin{itemize}
    \item[$\bullet$] 
    We derive a general intrinsic counterpart of the Bayesian Cramér-Rao bound, i.e., a performance bound for estimation problems where the parameter to estimate lies in a Riemannian manifold (with any arbitrary chosen Riemannian metric) and follows an assumed prior distribution. 
    To the best of our knowledge, few works addressed this topic: \cite{Jupp2010} derived a scalar inequality based on the Bayesian Cramér–Rao bound for smooth loss functions on manifolds, and \cite{Kumar2018} considered re-obtaining the Bayesian Cramér-Rao bound from the information theory perspective\cite{amari2000methods} i.e. the Riemannian geometry of the parameter space induced by metrics related to the statistical model (notably, the Fisher information metric is derived from the Hessian of the Kullback-Leibler divergence).
    Closer to the formalism considered in this paper (i.e., defining the error vector from the Riemannian logarithm), intrinsic Cramér-Rao bounds \cite{Bonnabel2015}, variants \cite{labsir2024intrinsic}, and intrinsic Bayesian Cramér-Rao bounds \cite{Labsir2022, labsir2023barankin} were proposed for parameter that belong to Lie groups (smooth manifolds with an additional group structure).
    
    \item[$\bullet$]
    We study the problem of covariance matrix estimation when the data is sampled from a Gaussian distribution, and whose covariance matrix is drawn from an inverse Wishart distribution \cite{barnard2000modeling}. 
    The inverse Wishart distribution being the conjugate prior of the multivariate Gaussian distribution, it has been leveraged in numerous references from the statistics \cite{barnard2000modeling, BOURIGA2013795, Fraley2007, Pourahmadi} and signal processing \cite{guerci2006knowledge, besson07, besson2008bounds, Bandiera2010, Bandiera2010Tyler, de2010knowledge, bidon2011knowledge, wang2017cfar} literature.
    Thanks to the obtained intrinsic Bayesian Cramér-Rao bound, performance bounds for this problem are obtained for both the mean squared error (Euclidean metric) and the natural Riemannian distance for Hermitian positive definite matrices (affine invariant metric). Both results are new to the best of our knowledge.  
    As in \cite{smith05}, the study of a proper geometry for the parameter space brings an interesting perspective to this problem.
  In the deterministic case, it was shown that the maximum likelihood estimator is biased and not efficient when investigating the affine invariant metric, while being apparently unbiased and efficient when evaluating error with the Euclidean metric. 
    Conversely in the Bayesian case, we observe that the maximum a posteriori and the minimum mean square error estimators are asymptotically efficient in the intrinsic case, while these appear not to be in the Euclidean one.

\end{itemize}

The rest of the article is organized as follows: 
Section \ref{sec:basic inequalities} contains the background related to Bayesian and intrinsic Cramér-Rao bounds, as well as necessary elements of Riemannian geometry.
Section \ref{sec:natural crb} presents the derivation of the intrinsic Bayesian Cramér-Rao bound.
Section \ref{sec:covariance crb} is dedicated to the analysis of intrinsic Bayesian Cramér-Rao bounds for the aforementioned covariance matrix estimation problem.
Finally, section \ref{sec:num_exp} shows simulations in which the obtained results allows to exhibit interesting properties of the Bayesian estimators.

Notations: Italic type indicates a scalar quantity, lower case boldface indicates a vector quantity, and upper case boldface a matrix. The transpose conjugate operator is $^H$ and the conjugate one is $^*$. $\tr(\cdot)$ and $|\cdot|$ are respectively the trace and the determinant operators. $\mathcal{H}^{++}_{p}$ is the manifold of symmetric positive definite (SPD) matrices of size $p \times p$. 
The notation $\mathbf{A} \succeq \mathbf{B}$ means that $\mathbf{A}-\mathbf{B}$ is positive definite.
A complex-valued random Gaussian vector of mean $\boldsymbol{\nu}$ and covariance matrix $\boldsymbol{\Sigma}$ is denoted $\mathbf{x} \sim \mathcal{CN} (\boldsymbol{\nu}, \boldsymbol{\Sigma})$.
\section{Bayesian and intrinsic Cramér-Rao lower bounds}
\label{sec:basic inequalities}

In this background section, we first recall the classical inequalities allowing to obtain the Cramér-Rao lower bound and its Bayesian counterpart. 
We then present elements of Riemaniann geometry and the corresponding 
intrinsic Cramér-Rao lower bound, which is a generalization of the Cramér-Rao inequality obtained for a parameter lying in a Riemannian manifold \cite{smith05}.

\subsection{Bayesian Cramér-Rao lower bound (Van Trees inequality)}

Let $\data\in\complexSpace^{\nfeatures}$ be an observation vector depending on an unknown parameter vector\footnote{Notice that if the parameter vector has several complex-valued entries, it is always possible to split their estimation problem in terms of real and imaginary parts, so we can assume without loss of generality that $\point$ is a real vector (with properly adjusted dimension $\dimension$).} $\point\in\realSpace^{\dimension}$ through its likelihood function denoted $f(\data|\point)$. 
When the unknown parameter vector $\point$ is assumed to be deterministic, the Cramér-Rao lower bound \cite{Rao1992,Cramer1946} states that any unbiased estimator $\estimator$ satisfies
\begin{equation}
    \mathbb{E}_{\data|\point}[(\estimator-\point)(\estimator-\point)^T] \succeq \FisherMat_{\point}^{-1},
\label{eq:eucl_crlb}
\end{equation}
where
\begin{equation}
    \FisherMat_{\point} = \mathbb{E}_{\data|\point}\left[\frac{\partial \log f(\data|\point)}{\partial \point}\frac{\partial \log f(\data|\point)}{\partial \point^T}\right]
\label{eq:eucl_fisher}
\end{equation}
is the Fisher information matrix (assumed to exists and to be non-singular throughout this paper).
This inequality can be reinterpreted in terms of Euclidean distance between the estimator and the parameter, as
\begin{equation}
    \mathbb{E}_{\data|\point}[d_E^2(\estimator,\point)] \geq \tr(\FisherMat_{\point}^{-1}),
\end{equation}
where $d_E(\estimator,\point)=\|\estimator-\point\|$ is the standard Euclidean distance.

In a Bayesian setting, the parameter vector is assumed to follow a prior distribution denoted $f_{\textup{prior}}(\point)$ and the joint distribution of the couple $\{\data,\point\}$ is denoted $f(\data,\point)=f(\data|\point)f_{\textup{prior}}(\point)$.
The score function is defined as
\begin{equation}
    s(\data,\point) = \frac{\partial \log f(\data,\point)}{\partial \point}
    \label{eq:bayes_score}
\end{equation}
and the corresponding Bayesian Fisher information matrix is defined as
\begin{equation}
    \FisherMat_{\textup{B}} = \mathbb{E}_{\data,\point} \left[s(\data,\point)s(\data,\point)^T\right].
\end{equation}
Since the joint log-likelihood can be expressed $\log f(\data,\point) = \log f(\data | \point) + \log f_{\textup{prior}}(\point)$, we can split the score function~\eqref{eq:bayes_score} into two components.
This results in a decomposition of the Bayesian Fisher information matrix as
\begin{equation}
    \FisherMat_{\textup{B}} = \mathbb{E}_{\point}\left[\FisherMat_{\point}\right]+\FisherMat_{\textup{prior}},
\label{eq:FisherMat_Bayes_Eucl}
\end{equation}
where
\begin{equation}
    \FisherMat_{\textup{prior}} = \mathbb{E}_{\point}\left[\frac{\partial \log f_{\textup{prior}}(\point)}{\partial \point}\frac{\partial \log f_{\textup{prior}}(\point)}{\partial \point^T}\right].
    \label{eq:Fisher_B_classic}
\end{equation}
From these definitions, a Bayesian counterpart of the Cramér-Rao lower bound, referred to as the Van Trees inequality~\cite{vantrees2001,vantrees2007}, is stated in the following theorem.

\begin{theorem}[Van Trees inequality~\cite{vantrees2001,vantrees2007}] 
Let $\estimator$ be an estimator of $\point$, and $\point_i $ (resp. $\estimator_i $) denote the $i^{\rm th}$ element of 
${\point}$ (resp. $\estimator$).
From~\cite{Weinstein88}, if $f(\data,\point)$ is absolutely continuous with respect to $\point$ a.e. $\data$, if $\FisherMat_{\textup{B}}$ in \eqref{eq:Fisher_B_classic} exists and is non-singular, and if the assumption 
\begin{equation}
    \lim_{\point_i\to\pm\infty} \point_i f(\data,\point) = 0
\label{eq:hypo_bayes_crb}
\end{equation}
holds for all element $i\in [\![1,\dimension]\!]$ and for any (a.e.) $\data\in\complexSpace^p$, then
\begin{equation}
    \mathbb{E}_{\data,\point}[(\estimator-\point)(\estimator-\point)^T] \succeq \FisherMat_{\textup{B}}^{-1},
\label{eq:bayescrb}
\end{equation}
which translates in expected Euclidean distance as 
\begin{equation}
    \mathbb{E}_{\data,\point}[d_E(\estimator,\point)^2] \geq \tr\left(\left(\mathbb{E}_{\point}\left[\FisherMat_{\point}\right]+\FisherMat_{\textup{prior}}  \right)^{-1}\right).
\label{eq:tracebayescrb}
\end{equation}
\label{thm:vantreesBCRB}
\end{theorem}

\subsection{Intrinsic Cramér-Rao bound}
\label{sec:background icrb}
\textit{Intrinsic} Cramér-Rao lower bound refers to the seminal work of~\cite{smith05}, that derived a generalization of the Cramér-Rao inequality~\eqref{eq:eucl_crlb} in the case where $\point$ is assumed to be deterministic and to lie in a Riemannian manifold.
To introduce this result, we first present some notations and tools of Riemannian geometry.
The inequality of~\cite{smith05} then lies in a careful transposition of the notion of estimation error vector to this context.
For more detailed coverage of differential geometry, one can refer to the standard textbooks on the topic~\cite{gallot1990riemannian,lang2012differential,lee2006riemannian}. 
The notations in this article are mostly inspired from the books~\cite{absil2009optimization,boumal2023introduction}, which provide very good (optimization-oriented) entry points to smooth manifolds and Riemannian geometry. 
Lastly, a more detailed introduction to information geometry and intrinsic Cramér-Rao bound is proposed in~\cite{bouchard2023fisher}.

\subsubsection{Riemannian geometry}
A smooth manifold $\manifold$ is a space that is locally diffeomorphic to a vector space and that possesses a differential structure.
Hence, each point $\point\in\manifold$ admits a tangent space $\tangentSpace{\point}$.
Tangent vectors in $\tangentSpace{\point}$ generalize directional derivatives at $\point$ on $\manifold$.
To turn $\manifold$ into a Riemannian manifold, it needs to be equipped with a Riemannian metric $\metric{\cdot}{\cdot}{\cdot}$, which defines smoothly varying inner products on tangent spaces.
It is important to notice that the definition of the metric is a choice, that will induce -- along with the geometrical constraints -- a corresponding geometry for the manifold $\manifold$.
% The metric allows notably for locally defining the notion of angle and length for vectors in $T_{\boldsymbol{\theta}}\mathcal{M}$.
%
Figure \ref{fig:manifold} illustrates the aforementioned spaces and objects.

\begin{figure}[!t]
    \centering
    \begin{center}

\begin{tikzpicture}[scale=2]
% half sphere
\draw  (1,0) arc (30:150:1.155)
      plot [smooth, domain=pi:2*pi] ({cos(\x r)},{0.2*sin(\x r)});
\draw [dotted] plot [smooth, domain=0:pi] ({cos(\x r)},{0.2*sin(\x r)});
% name manifold
\draw (0,-0.3) node {\small$\manifold$};

% point x
\coordinate (x) at (-0.5,0.25);
% tangent space
\draw [fill=gray!20,opacity=0.4] (x)++(0.5,0.15) -- ++(-0.4,-0.5) -- ++(-0.5,+0.25) -- ++(+0.4,+0.5) -- cycle;
% name tangent space x
\draw [opacity=0.8] (-0.9,0.6) node {\small$\tangentSpace{\point}$};
% x
\draw (x)++(-0.25,0.1) node[anchor=north west] {\footnotesize$\point$};

\draw[draw=myorange,fill=myorange,line width=1pt,>=stealth,->] (x) -- (-0.3,0.5) node[near end,anchor=east,font=\footnotesize,color=myorange] {$\tangent$};

\draw[draw=myorange,fill=myorange,line width=1pt,>=stealth,->] (x) -- (-0.3,0.1) node[midway,anchor=north,font=\footnotesize,color=myorange] {$\tangentBis$};

\draw[draw=myred] (-0.44,0.32) to[bend left=45] (-0.42,0.2);
\node[font=\footnotesize,color=myred] at (-0.3,0.27) {$\alpha$};

\draw (x) node {\tiny$\bullet$};

\node[font=\scriptsize,color=gray] at (4,0.3)
{
$
    \begin{array}{l c l}
         \textup{Inner product on } \tangentSpace{\point} & ~ & \metric{\point}{\cdot}{\cdot} : \tangentSpace{\point}\times\tangentSpace{\point}\to\realSpace \vspace{0.2cm} \\
         \textup{Norm of } \tangent \in \tangentSpace{\point} & ~ &  \|\tangent\|_{\point}^2 = \metric{\point}{\tangent}{\tangent} \vspace{0.2cm} \\
         \textup{Angle}~\alpha~\textup{between } \tangent, \tangentBis \in \tangentSpace{\point} & ~ & \alpha(\tangent, \tangentBis) = \arccos \frac{\metric{\point}{\tangent}{\tangentBis}}{\|\tangent\|_{\point}\|\tangentBis\|_{\point}}
    \end{array}
$
};

\end{tikzpicture}
\end{center}
    \caption{A visual representation of a smooth manifold $\manifold$, its tangent space $\tangentSpace{\point}$ at point $\point$, and two tangent vectors $\tangent$ and $\tangentBis$.
    The metric $\metric{\point}{\cdot}{\cdot}$ is an an inner product on $\tangentSpace{\point}$ and it induces the notion of length and angle.}
    \label{fig:manifold}
\end{figure}

\paragraph{Second-order differentiation on a Riemannian manifold}
Vector fields are functions associating one tangent vector to each point on $\manifold$.
The directional derivatives of vector fields (i.e., the generalization of the standard directional derivatives) is defined through the Levi-Civita connection $\LC{\cdot}{\cdot}$.
Such generalization is indeed needed because the tangent space changes when one moves from one point to another on a manifold.
Hence, the Levi-Civita connection provides a way of differentiating that properly reflects the structure of the manifold -- i.e., the effect of the geometrical constraints and chosen metric -- as illustrated in Figure~\ref{fig:connection}.
In practice, the Levi-Civita connection is obtained thanks to the Koszul formula: given vector fields $\tangent$, $\tangentBis$ and $\tangentTer$ at $\point\in\manifold$, it is
\begin{multline*}
    2\metric{\point}{\LC{\tangent}{\tangentBis}}{\tangentTer} =
    \Diff\metric{\point}{\tangentBis}{\tangentTer}[\tangent]
    +\Diff\metric{\point}{\tangent}{\tangentTer}[\tangentBis]
    -\Diff\metric{\point}{\tangent}{\tangentBis}[\tangentTer]
    \\
    + \metric{\point}{\tangentTer}{[\tangent,\tangentBis]}
    + \metric{\point}{\tangentBis}{[\tangentTer,\tangent]}
    - \metric{\point}{\tangent}{[\tangentBis,\tangentTer]},
\end{multline*}
where $\Diff\cdot[\cdot]$ denotes the directional derivative and $[\cdot,\cdot]$ the Lie bracket.
\begin{figure}[!t]
    \centering
    \begin{center}
\scriptsize
\begin{tikzpicture}[scale=1.4]

\draw[line width=0.7pt] (0,0) -- (2,0);

\draw[color=myblue, line width=1pt, ->, >=stealth] (0,0) -- ++(0.8,0) node[near end, below] {$\tangent$};
\draw[color=myorange, line width=1pt, ->, >=stealth] (0,0) -- ++(0.4,0.7) node[very near end, above, anchor=south east] {$\tangentBis$};
\draw[color=myred, line width=1pt, ->, >=stealth] (0,0) -- ++(0.3,-0.3) node[at end, below, anchor=north] {$\Diff\tangentBis[\tangent]$};
\draw[color=myorange, dashed, line width=0.7pt, ->, >=stealth] (0,0) -- ++(0.7,0.4);
\draw[color=myred, dashed, line width=0.7pt, ->, >=stealth] (0.4,0.7) -- ++(0.3,-0.3);
\node at (0,0) {\tiny$\bullet$};
\node[anchor=north east] at (0,0) {$\point$};

\draw[color=myblue, line width=1pt, ->, >=stealth] (2,0) -- ++(0.8,0) node[near end, below] {$\tangentAlt$};
\draw[color=myorange, dashed, line width=0.7pt, ->, >=stealth] (2,0) -- ++(0.4,0.7);
\draw[color=myorange, line width=1pt, ->, >=stealth] (2,0) -- ++(0.7,0.4) node[at end, below, anchor=west] {$\tangentAltBis$};
\node at (2,0) {\tiny$\bullet$};
\node[anchor=north east] at (2,0) {$\pointAlt$};

\begin{scope}[shift={(5.5,-0.4)},scale=2.1]
    \draw  (1,0) arc (30:150:1.155)
      plot [smooth, domain=pi:2*pi] ({cos(\x r)},{0.2*sin(\x r)});
    \draw [dotted] plot [smooth, domain=0:pi] ({cos(\x r)},{0.2*sin(\x r)});
\end{scope}

\begin{scope}[shift={(4.5,0)}]
    
    \draw[line width=0.7pt] (0,-0.2) to[bend right=17] (1.6,-0.25);

    \draw[color=myblue, line width=1pt, ->, >=stealth] (0,-0.2) -- ++(0.5,-0.2) node[near end, below, anchor=north west] {$\tangent$};
    \draw[color=myorange, line width=1pt, ->, >=stealth] (0,-0.2) -- ++(0.27,0.45) node[near end, above, anchor=east] {$\tangentBis$};
    \draw[color=myorange, dashed, line width=0.7pt, ->, >=stealth] (0,-0.2) -- ++(0.47,0.19);
    \draw[color=myred, line width=1pt, ->, >=stealth] (0,-0.2) -- ++(0.2,-0.26) node[near end, below, anchor=north east] {$\LC{\tangent}\tangentBis$};
    \draw[color=myred, dashed, line width=0.7pt, ->, >=stealth] (0.27,0.25) -- ++(0.2,-0.26);
    \node at (0,-0.2) {\tiny$\bullet$};
    \node[anchor=east] at (0,-0.2) {$\point$};

    \draw[color=myblue, line width=1pt, ->, >=stealth] (1.6,-0.25) -- ++(0.5,0.17) node[midway, below, anchor=north west] {$\tangentAlt$};

    \draw[color=myorange, dashed, line width=0.7pt, ->, >=stealth] (1.6,-0.25) -- ++(-0.1,0.52);

    \draw[color=myorange, line width=1pt, ->, >=stealth] (1.6,-0.25) -- ++(0.22,0.4) node[at end, below, anchor=west] {$\tangentAltBis$};
    
    \node at (1.6,-0.25) {\tiny$\bullet$};
    \node[anchor=north east] at (1.6,-0.25) {$\pointAlt$};
\end{scope}

\end{tikzpicture}
\end{center}
    \caption{
    Illustration of directional derivative $\Diff\tangentBis[\tangent]$ (left) and Levi-Civita connection $\LC{\tangent}{\tangentBis}$ (right) of a vector field $\tangentBis$ in the direction $\tangent$ at $\point$.
    As the directional derivative, an Levi-Civita connection describes how the vector field $\tangentBis$ evolves in a given direction $\tangent$.
    In addition, the affine connection takes into account the structure of the manifold (curvature, and non-constant metric).
    }
    \label{fig:connection}
\end{figure}

\paragraph{Riemannian curvature}
The Levi-Civita connection allows to define the corresponding Riemannian curvature tensor, which is
\begin{equation}
    \mathcal{R}(\tangent,\tangentBis)\tangentTer = 
    \LC{\LC{\tangentTer}{\tangentBis}}{\tangent}
    - \LC{\LC{\tangentTer}{\tangent}}{\tangentBis}
    - \LC{\tangentTer}{[\tangent,\tangentBis]}.
\label{eq:def_tensorcurvature}
\end{equation}
It measures how second-order differentiation fails to commute on the Riemannian manifold $\manifold$.
Indeed, it is the difference between $\LC{\LC{\cdot}{\tangentBis}}{\tangent}-\LC{\LC{\cdot}{\tangent}}{\tangentBis}$ and $\LC{\cdot}{[\tangent,\tangentBis]}$.
As for the Levi-Civita connection, it reflects the effects of both the geometrical constraints and the (non-constant) Riemannian metric.

\paragraph{Geodesics and Riemannian distance}  \label{sec:geodesics}
The Levi-Civita connection also yields geodesics on $\manifold$, which generalize the concept of straight lines.
The geodesic $\gamma:[0,1]\to\manifold$ such that $\gamma(0)=\point$ and $\dot{\gamma}(0)=\tangent$ is defined through the differential equation
\begin{equation}
    \nabla_{\dot{\gamma}(t)} \dot{\gamma}(t) = \boldsymbol{0},
\end{equation}
and are illustrated in Figure \ref{fig:geodesic_exp_log}.
From there, one can define the Riemannian exponential mapping.
At $\point$, it is the mapping $\exp_{\point}:\tangentSpace{\point}\to\manifold$ such that, for all $\tangent\in\tangentSpace{\point}$, $\exp_{\point}(\tangent)=\gamma(1)$. 
%where $\gamma$ is the geodesic such that $\gamma(0)=\point$ and $\dot{\gamma}(0)=\tangent$.
%
Its inverse, the Riemannian logarithm is, at $\point\in\manifold$, the mapping $\log_{\point}:\manifold\to\tangentSpace{\point}$ such that, for $\pointAlt\in\manifold$, $\log_{\point}(\pointAlt)=\tangent$ where $\exp_{\point}(\tangent)=\pointAlt$.
The Riemannian exponential and logarithm mappings are illustrated in Figure \ref{fig:geodesic_exp_log}.
The manifold $\mathcal{M}$ is said to be complete if for all $\theta\in\mathcal{M}$ the Riemannian exponential is well defined on the entire tangent space $T_\theta \mathcal{M}$.
Otherwise we can restrict this operator to be defined at each point only for tangent vectors whose norm are lower than the so-called injectivity radius (largest radius for which the exponential map at $\theta$ is a diffeomorphism). 
In any case, the Riemannian logarithm $\log_\theta$ might not be uniquely defined (e.g., multiple windings for great-circle geodesics on the sphere). 
In some cases, one can obtain a unique definition for this operator by choosing the tangent vector of shortest length.
In this study, we assume that $\exp_\theta$ and $\log_\theta$ are well defined on the manifold.
As explained in~\cite{smith05}, this can be extended to the case of conjugate points for which $\exp_\theta$ is singular because such points have measure zero.
Finally, the Riemannian distance $d_R:\manifold\times\manifold\to\realSpace$ is
\begin{equation}
    d_R(\point,\pointAlt) = \int_0^1 \sqrt{\langle\dot{\gamma}(t),\dot{\gamma}(t)\rangle_{\gamma(t)}} \,\; dt,
\end{equation}
where $\gamma$ is the geodesic on $\manifold$ such that $\gamma(0)=\point$ and $\gamma(1)=\pointAlt$.

\begin{figure}[!t]
    \centering
    \begin{tikzpicture}[scale=2.3]
% half sphere
\draw  (1,0) arc (30:150:1.155)
      plot [smooth, domain=pi:2*pi] ({cos(\x r)},{0.2*sin(\x r)});
\draw [dotted] plot [smooth, domain=0:pi] ({cos(\x r)},{0.2*sin(\x r)});

\coordinate (x) at (-0.495,0.2);
\coordinate (y) at (0.4,0.2);

\draw[draw=myblue, line width=1pt] (x) to[bend left = 20] (y);
\node at (x) {\scriptsize$\bullet$};
\node[anchor=north] at (x) {\scriptsize$\point$};
\node at (y) {\scriptsize$\bullet$};
\node[anchor=north] at (y) {\scriptsize$\pointAlt$};

\coordinate (z) at (-0.05,0.29);

\draw[draw=myorange, line width=1pt,->,>=stealth] (z)++(0,0.005) -- ++(0.38,0) node[near end, above] {\scriptsize\color{myorange}$\dot{\gamma}(t)$};

\node at (z) {\scriptsize\color{myblue}$\bullet$};
\node[anchor=north] at (z) {\scriptsize\color{myblue}$\gamma(t)$};

\begin{scope}[shift={(2.5,0)}]
    % half sphere
    \draw  (1,0) arc (30:150:1.155)
          plot [smooth, domain=pi:2*pi] ({cos(\x r)},{0.2*sin(\x r)});
    \draw [dotted] plot [smooth, domain=0:pi] ({cos(\x r)},{0.2*sin(\x r)});

    \coordinate (x) at (-0.495,0.2);
    \coordinate (y) at (0.4,0.2);
    
    % tangent space
    \draw [fill=gray!20,opacity=0.4] (x)++(0.8,0.2) -- ++(-0.8,-0.6) -- ++(-0.6,+0.3) -- ++(+0.8,+0.6) -- cycle;

    \draw[draw=myblue,line width=1pt] (x) to[bend left = 20] (y);

    \draw[draw=myred,dashed,line width=0.8pt] (x)++(0.5,0.2) -- (y);

    \draw[draw=myorange, line width=1pt,->,>=stealth] (x)++(0,0.005) -- ++(0.5,0.2) node[near start, above, sloped] {\scriptsize\color{myorange}$\tangent=\log_{\point}(\pointAlt)$};
    
    \node at (x) {\scriptsize$\bullet$};
    \node[anchor=north] at (x) {\scriptsize$\point$};
    \node at (y) {\scriptsize$\bullet$};
    \node[anchor=north] at (y) {\scriptsize$\pointAlt=\exp_{\point}(\tangent)$};
\end{scope}

\end{tikzpicture}
    \caption{Illustration of geodesics (left), Riemannian exponential and logarithm mappings (right).
    The Riemannian distance $d_R(\point,\pointAlt)$ is the length of the geodesic joining $\point$ and $\pointAlt$.}
    \label{fig:geodesic_exp_log}
\end{figure}
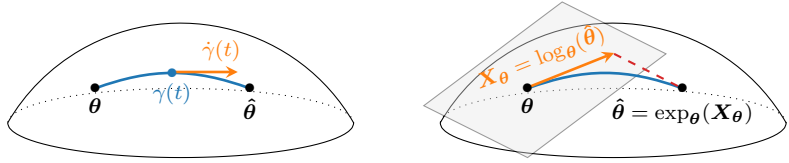

\subsubsection{From Riemannian geometry to performance bound}
We consider the problem of estimating a parameter $\point$ in a manifold $\manifold$ from some data $\data$ in $\complexSpace^{\nfeatures}$.
In the considered setting, the data at hand are drawn from a distribution with probability density function $f(\data|\point)$.
As for the classical Euclidean case, one wishes to establish a lower performance bound for the estimation problem.
This is achieved by the so-called intrinsic Cramér-Rao lower bound~\cite{smith05, boumal2014thesis}.

\paragraph{Estimation error}
First, one needs to re-define the estimation error of an unbiased estimator $\estimator\in\manifold$ of the true parameter $\point\in\manifold$.
Indeed, the notion of subtraction ``$\estimator-\point$'' from the Euclidean case described in Section \ref{sec:basic inequalities} is not defined intrinsically on $\manifold$.
A proper definition of the error that is in accordance with the manifold $\manifold$ and the geometry induced by the chosen Riemannian metric $\metric{\cdot}{\cdot}{\cdot}$ is provided by the Riemannian logarithm: an element of the tangent space $\tangentSpace{\point}$ of $\point$ that ``points towards'' $\estimator$, and whose norm corresponds to the Riemannian distance $d_R(\point,\pointAlt)$ (cf. Section  and Section \ref{sec:geodesics} and Figure \ref{fig:geodesic_exp_log}).
Furthermore, it will be helpful to handle this object using a system of coordinates \cite{Bonnabel2015}.
To do so, let $\{\basisElement_i\}_{i=1}^{\dimension}$ be an orthonormal basis of $\tangentSpace{\point}$, where $d$ is the dimension of $\manifold$.
In practice, such a basis can be obtained either analytically from mathematical calculations or numerically, thanks to the Gram-Schmidt orthonormalization process.
This basis yields the decomposition
\begin{equation}
    \log_{\point}(\estimator) = \sum_{i=1}^{\dimension} (\errVec)_i \basisElement_i.
\end{equation}
The vector $\errVec=[(\errVec)_i]_{i=1}^{\dimension}\in\realSpace^{\dimension}$ is the coordinates error vector, which is obtained as
\begin{equation}
    (\errVec)_i = \metric{\point}{\log_{\point}(\estimator)}{\basisElement_i}.
\label{eq:errVec}
\end{equation}
The covariance matrix of this error measure is $\errCov = \mathbb{E}_{\data}[\errVec\errVec^T]$.
Further note that the squared norm of this vector $\errVec$ -- which is the trace of $\errCov$ -- corresponds to the squared Riemannian distance between $\point$ and $\estimator$, i.e.,
\begin{equation}
    d_R^2(\point,\estimator) = \metric{\point}{\log_{\point}(\estimator)}{\log_{\point}(\estimator)} = \| \errVec \|_2^2 = \tr(\errCov).
    \label{eq:link_dist_log}
\end{equation}

\paragraph{Fisher information matrix}
To obtain an inequality such as~\eqref{eq:eucl_crlb}, we need to redefine the Fisher information matrix as follows: it is the matrix $\FisherMat_{\point}\in\realSpace^{\dimension\times\dimension}$ whose $ij^{\textup{th}}$ element is
\begin{equation}
    (\FisherMat_{\point})_{ij} = \metric{\point}{\basisElement_i}{\basisElement_j}^{\textup{FIM}},
\label{eq:Def_Ftheta}
\end{equation}
where $\metric{\point}{\cdot}{\cdot}^{\textup{FIM}}$ is the Fisher information metric.
As explained in~\cite{smith05}, this metric is defined, for all $\point\in\manifold$, $\tangent,\tangentBis\in\tangentSpace{\point}$~as
\begin{equation}
    \metric{\point}{\tangent}{\tangentBis}^{\textup{FIM}}
    = \mathbb{E}_{\data}[\Diff_{\point} \log f(\data|\point)[\tangent]\cdot \Diff_{\point} \log f(\data|\point)[\tangentBis]]
    = -\mathbb{E}_{\data}[\Diff^2_{\point} \log f(\data|\point)[\tangent,\tangentBis]],
\end{equation}
where $\Diff^2\cdot[\cdot,\cdot]$ is the second order directional derivative%
\footnote{
    When differentiating functions with several variables such as $f(\data|\point)$, we indicate the variable that is differentiated with a subscript, i.e., $\Diff_{\point} \log f(\data|\point)[\cdot]$ is the directional derivative of $\log f(\data|\point)$ with respect to the variable $\point$.
}.

\begin{remark}
    Notice that we consider a geometry and notion of estimation error induced by any chosen metric $\metric{\cdot}{\cdot}{\cdot}$.
    In particular, the geometry induced by the Fisher information metric $\metric{\cdot}{\cdot}{\cdot}^{\textup{FIM}}$ is referred to as the Fisher-Rao information geometry (of the considered statistical model) \cite{bouchard2023fisher}.
    Interestingly, the choice $\metric{\cdot}{\cdot}{\cdot}=\metric{\cdot}{\cdot}{\cdot}^{\textup{FIM}}$ always yields $\FisherMat_{\point}=\eye_{\dimension}$, which underlines the inherent match between the Fisher information metric with the estimation problem.
    However, the Fisher-Rao information geometry might not always be tractable, or might not always be of interest for measuring the estimation performance in some contexts.
    Hence, being able to rely on any arbitrary metric $\metric{\cdot}{\cdot}{\cdot}$ is a key point in this work.
\end{remark}

%Notice that if $\metric{\cdot}{\cdot}{\cdot}=\metric{\cdot}{\cdot}{\cdot}^{\textup{FIM}}$, then one simply have 
%
%It thus appears very advantageous to choose the Fisher information metric and the resulting geometry on the manifold $\manifold$. \textcolor{magenta}{ab: je vais retaper ici un peu, l'important c'est plus d'insister que "attention l'erreur et la géométrie est avec $<,>$ et pas $FIM$"}
%Unfortunately, it is not always possible to use it because some geometrical tools -- which can be very difficult or even impossible to obtain -- remain unknown.

\paragraph{Intrinsic Cramér-Rao lower bound}
The Fisher information matrix $\FisherMat_{\point}$ can be leveraged to define a performance lower bound on the covariance $\errCov$ of the coordinates error vector $\errVec$, it yields the inequality stated in Theorem~\ref{thm:ICRB}.

\begin{theorem}[intrinsic Cramér-Rao lower bound~\cite{smith05}]
\label{thm:ICRB}
    Let $\estimator$ be an unbiased estimator of $\point$ in $\manifold$.
    The intrinsic Cramér-Rao lower bound of the covariance $\errCov=\mathbb{E}_{\data}[\errVec\errVec^T]$ of the error vector defined in~\eqref{eq:errVec}~is
    \begin{equation}
        \errCov \succeq \FisherMat_{\point}^{-1} - \frac13(\FisherMat_{\point}^{-1}\mathcal{R}_m(\FisherMat_{\point}) + \mathcal{R}_m(\FisherMat_{\point})\FisherMat_{\point}^{-1}),
    \label{eq:icrb}
    \end{equation}
    where the Fisher matrix $\FisherMat_{\point}$ is defined in~\eqref{eq:Def_Ftheta} and $\mathcal{R}_m$ is a curvature related term such that, for $i,j\in\{1,\dots,\dimension\}$,
    \begin{equation}
        (\mathcal{R}_m(\errCov))_{i,j} = \mathbb{E}_{\data}[\metric{\point}{\mathcal{R}(\log_{\point}(\estimator),\basisElement_i)\basisElement_j}{\log_{\point}(\estimator)}].
    \end{equation}
    Taking the trace of this inequality further yields
    \begin{equation}
        \mathbb{E}_{\data}[d_R^2(\estimator,\point)] \geq
        \tr(\FisherMat_{\point}^{-1}  - \frac13(\FisherMat_{\point}^{-1}\mathcal{R}_m(\FisherMat_{\point}) + \mathcal{R}_m(\FisherMat_{\point})\FisherMat_{\point}^{-1})).
    \end{equation}
\end{theorem}

\begin{remark}[From intrinsic (Riemannian) to Euclidean Cramér-Rao bound]
    If we consider the manifold $\manifold$ as the Euclidean space $\realSpace^{\dimension}$ equipped with the Euclidean metric, then the Cramér-Rao inequality~\eqref{eq:icrb} from Theorem~\ref{thm:ICRB} is equal to~\eqref{eq:eucl_crlb}.
    Indeed, in this case, the logarithm mapping is simply $\log_{\point}(\estimator)=\estimator-\point$, hence the error vector is $\errVec=\estimator-\point$. The Fisher information matrix ~\eqref{eq:Def_Ftheta} is then reduced to definition~\eqref{eq:eucl_fisher}. Finally, $\realSpace^{\dimension}$ is flat and the curvature tensor~\eqref{eq:def_tensorcurvature} is equal to zero. 
\end{remark}

\section{Intrinsic Bayesian Cramér-Rao bound}
\label{sec:natural crb}

In this section, an intrinsic Bayesian Cramér-Rao lower bound is proposed.
More specifically, an intrinsic derivation of the Van-Trees inequality is presented.
It consists in the generalization of Theorem~\ref{thm:vantreesBCRB} for an error vector defined from objects on a manifold $\manifold$ as in Section~\ref{sec:background icrb}.
This section is divided in two subsections.
Section~\ref{subsec:ibcrb_hypotheses} contains the required assumptions.
Section~\ref{subsec:ibcrb_thm} presents the intrinsic Bayesian Cramér-Rao lower bound along with its proof.

\subsection{Hypotheses}
\label{subsec:ibcrb_hypotheses}

To derive the intrinsic Bayesian Cramér-Rao bound stated in Theorem~\ref{thm:main}, we first need several assumptions on the model.
These ensure that the computed derivatives, expectations, and matrix inversions are indeed well defined.
These assumptions are provided below, along with some explanations and interpretations.

\begin{assumption} $f(\data,\point)$ is absolutely continuous with respect to $\point$ for any $\data$.
\label{assum:regularity}
\end{assumption}
\noindent
Assumptions \ref{assum:regularity} simply state a standard regularity condition that is used when studying Bayesian bounds~(cf. \cite{Weinstein88}).

\begin{assumption} The parameterization of chosen orthonormal bases $\{\basisElement_i\}_{i=1}^{\dimension}$ of tangent spaces $\tangentSpace{\point}$ is smooth with respect to $\point$.
\label{assum:smoothbasis}
\end{assumption}
\noindent
Assumption~\ref{assum:smoothbasis} is simply technical and usually satisfied: bases are in practice constructed with canonical elements and smooth functions of $\point$ to ensure orthonormality with respect to the chosen metric $\metric{\cdot}{\cdot}{\cdot}$.

\begin{assumption}
The manifold $\manifold$ is measurable, and the function $\mu(\cdot)$ defines a measure for this space.
\label{assum:manifold}
\end{assumption}
\noindent
Assumption \ref{assum:manifold} allows to define the expectation over $\point$ as
\begin{equation}
    \mathbb{E}_{\point}[g(\point)] = \int_{\point\in\manifold} g(\point) f_{\textup{prior}}(\point) \mu({d\point}).
\label{eq:def_expectation_measure}
\end{equation}
Notice that, in the general case, the notion of integration over a manifold is not trivial, and conditions the design of $f_{\textup{prior}}$, that needs to integrate to one over $\manifold$.
In this setup, even computing the normalizing constant is often challenging.
Examples of such distributions include Wishart-based priors (cf. Section \ref{sec:covariance crb}), distributions on orthonormal frames and subspaces (cf. \cite{abdallah2020bayesian} and references therein), Riemannian Gaussian distribution~\cite{said2017riemannian}, and concentrated Gaussian distributions on Lie groups \cite{labsir2024intrinsic}.

\begin{assumption} $\FisherMat_{\textup{B}}$ defined for each element by $(\FisherMat_{\textup{B}})_{i,j}
        = \mathbb{E}_{\data,\point}[\Diff_{\point}\log f(\data,\point)[\basisElement_i] \cdot \Diff_{\point}\log f(\data,\point)[\basisElement_j]]$ exists and is non-singular.
\label{assum:regularity2}
\end{assumption}
\noindent
As for assumption \ref{assum:regularity}, assumption \ref{assum:regularity2} states a standard regularity condition that is used when studying Bayesian bounds~(cf. \cite{Weinstein88}).

\begin{assumption} 
For all $i,j\in\{1,\dots,\dimension\}$, the relationship
\begin{equation}
    \int_{\manifold} \int_{\complexSpace^{\nfeatures}}
    \metric{\point}{\LC{(f(\data,\point)\log_{\point}{\estimator})}{\basisElement_j}}{\basisElement_i} d\data \mu(d\point)
    = 0
\end{equation}
is satisfied.
\label{eq:condition3IBCRB}
\end{assumption}
\noindent
Assumption~\ref{eq:condition3IBCRB} also appears as rather technical, but is related to assumption~\eqref{eq:hypo_bayes_crb} in the usual derivation of the Bayesian Cramér-Rao bound.
Indeed a proof strategy of Theorem~\ref{thm:vantreesBCRB} involves an integration by parts, requiring to state 
\begin{equation}
    \int_{\realSpace^{\dimension}} \int_{\complexSpace^{\nfeatures}} \frac{\partial(\estimator-\point)f(\data,\point)}{\partial \point^T}d\data d\point = \mathbf{0} 
    \label{eq:integral_is_0_Eucl}
\end{equation}
before concluding\footnote{In \cite{Weinstein88}, the authors used assumption \eqref{eq:hypo_bayes_crb} only since their work is conducted for a special case of a parameter vector lying on $\mathbb{R}^d$.}.
%Additional arguments allowing to conclude are that $\estimator$ does not depend on $\point$ and that $\forall i$, $\lim_{\point_i\to\pm\infty} f(\data,\point) = 0$ for any $\data$ (which is also implicitly assumed in ~\cite{Weinstein88}).}}.
%
Assumption~\ref{eq:condition3IBCRB} is the Riemannian counterpart of~\eqref{eq:integral_is_0_Eucl}, i.e., it states the same condition when using the formalism of intrinsic Cramér-Rao bound detailed in Section~\ref{sec:natural crb}. 
Further notice that in the case where $\manifold$ is a unimodular Lie group~\cite{labsir2024intrinsic}, an equivalent condition to Assumption~\ref{eq:condition3IBCRB} is required and always satisfied -- see Equation~(11) in~\cite{labsir2024intrinsic}.

% In fact, Assumption \ref{eq:condition3IBCRB} also reduces to \eqref{eq:integral_is_0_Eucl} when \textcolor{magenta}{(cf remark XXX) : la conclusion de section II ou tout retombe sur la CRB quand metric is eucl and manifold is Rd}.
% %

% \textcolor{magenta}{+ Il faut les deux exemples pour la fin de ce paragraphe : c'est vrai sur les groupes de Lie, et ça se comprend bien quand on est dans un ouvert (ou le bord de M correspond à l'infini dans la limite classique)}

\subsection{Main theorem}
\label{subsec:ibcrb_thm}

We are now able to provide our main result, i.e., an intrinsic derivation of the Van Trees inequality, in Theorem~\ref{thm:main}.
As for the usual intrinsic Cramér-Rao bound, it extends the (Euclidean) Bayesian bound from Theorem~\ref{thm:vantreesBCRB} to the context of a Riemannian manifold $\manifold$.

\begin{theorem}[Intrinsic Bayesian Cramér-Rao lower bound]
\label{thm:main}
    Let $\estimator$ be an estimator of a parameter $\point\in\manifold$ constructed from $\data\in\complexSpace^{\nfeatures}$, with $\data\simeq f(\data|\point)$ and $\point\sim f_{\textup{prior}}(\point)$.
    Under assumptions \ref{assum:regularity}-\ref{eq:condition3IBCRB}, the covariance matrix of the error $\errCovBayes=\mathbb{E}_{\data,\point}[\errVec\errVec^T]$ satisfies the following inequality
    \begin{equation}
        \errCovBayes + \frac13( \FisherMat_{\textup{B}}^{-1} R_m(\errCovBayes) + R_m(\errCovBayes) \FisherMat_{\textup{B}}^{-1} ) \succeq \FisherMat_{\textup{B}}^{-1} ,
    \label{eq:ineq_Bayes}
    \end{equation}
    in which the Bayesian Fisher matrix $\FisherMat_{\textup{B}}$ is defined with
    \begin{equation}
        (\FisherMat_{\textup{B}})_{i,j}
        = \mathbb{E}_{\data,\point}[\Diff_{\point}\log f(\data,\point)[\basisElement_i] \cdot \Diff_{\point}\log f(\data,\point)[\basisElement_j]]
        = \mathbb{E}_{\data,\point}[-\Diff^2_{\point}\log f(\data,\point)[\basisElement_i,\basisElement_j]],
    \label{eq:FisherMat_Bayes}
    \end{equation}
    and where the curvature-related term $R_m$ defined from $\errCovBayes$ as
    \begin{equation}
        (R_m(\errCovBayes))_{i,j} = \mathbb{E}_{\data,\point}[\metric{\point}{\mathcal{R}(\log_{\point}(\estimator),\basisElement_i)\basisElement_j}{\log_{\point}(\estimator)}].
    \label{eq:curv_term_Bayes}
    \end{equation}    
\end{theorem}
\begin{proof}
    As in the Euclidean case, we start by defining the score $\scoreVec\in\realSpace^{\dimension}$, whose $i^{\textup{th}}$ element is defined as
    \begin{equation}
        (\scoreVec)_i = \Diff_{\point} \log f(\data,\point)[\basisElement_i].
    \end{equation}
    We also define the Fisher matrix as $\FisherMat_{\textup{B}}=\mathbb{E}_{\data,\point}[\scoreVec \scoreVec^T]$ -- also defined in~\eqref{eq:FisherMat_Bayes}.
    As for the Euclidean derivation of the Bayesian bound, we set
    \begin{equation}
        \boldsymbol{v}_{\point} = \errVec - \FisherMat_{\textup{B}}^{-1}\scoreVec,
    \end{equation}
    where $\errVec$ is the coordinates error vector defined in~\eqref{eq:errVec}.
    By definition, we have $\mathbb{E}_{\data,\point}[\boldsymbol{v}_{\point} \boldsymbol{v}_{\point}^T] \succeq \boldsymbol{0}$.
    Hence,
    \begin{equation}
        \mathbb{E}_{\data,\point}[\boldsymbol{v}_{\point} \boldsymbol{v}_{\point}^T] =
        \errCovBayes + \FisherMat_{\textup{B}}^{-1}
        - \FisherMat_{\textup{B}}^{-1}\mathbb{E}_{\data,\point}[\scoreVec\errVec^T]
        - \mathbb{E}_{\data,\point}[\errVec\scoreVec^T]\FisherMat_{\textup{B}}^{-1}
        \succeq \boldsymbol{0},
    \label{eq:inequality_start}
    \end{equation}
    where $\errCovBayes=\mathbb{E}_{\data,\point}[\errVec\errVec^T]$.
    As in the Euclidean case, the key issue is the computation of $\mathbb{E}_{\data,\point}[\scoreVec\errVec^T]$ (or equivalently $\mathbb{E}_{\data,\point}[\errVec\scoreVec^T]$).
    Let $\boldsymbol{A}=\mathbb{E}_{\data,\point}[\errVec\scoreVec^T]$.
    Each element $(i,j)\in\{1,\dots,\dimension\}^2$ is defined as
    \begin{equation}
        \boldsymbol{A}_{i,j} = \int_{\manifold}\int_{\complexSpace^{\nfeatures}}
        \metric{\point}{\log_{\point}(\estimator)}{\basisElement_i}
        \Diff_{\point} \log f(\data,\point)[\basisElement_j]
        f(\data,\point) d\data \mu(d\point).
    \end{equation}
    Since
    \begin{equation}
        \Diff_{\point} \log f(\data,\point)[\basisElement_j] = \frac{1}{f(\data,\point)} \Diff_{\point} f(\data,\point)[\basisElement_j],
    \end{equation}
    we get
    \begin{equation}
        \begin{array}{rcl}
            \boldsymbol{A}_{i,j} & = &
            \int_{\manifold}\int_{\complexSpace^{\nfeatures}}
            \metric{\point}{\log_{\point}(\estimator)}{\basisElement_i}
            \Diff_{\point} f(\data,\point)[\basisElement_j]
            d\data \mu(d\point)
            \\
            & = &
            \int_{\manifold}\int_{\complexSpace^{\nfeatures}}
            \metric{\point}{\Diff_{\point} f(\data,\point)[\basisElement_j]\log_{\point}(\estimator)}{\basisElement_i}
            d\data \mu(d\point).
        \end{array}
    \end{equation}
    Since the Levi-Civita connection satisfies Leibniz rule, we have
    \begin{equation}
        \metric{\point}{\LC{(f(\data,\point)\log_{\point}(\estimator))}{\basisElement_j}}{\basisElement_i} =
        \metric{\point}{\Diff_{\point}f(\data,\point)[\basisElement_j]\log_{\point}(\estimator)}{\basisElement_i}
        + \metric{\point}{f(\data,\point)\LC{\log_{\point}(\estimator)}{\basisElement_j}}{\basisElement_i}.
    \end{equation}
    From there, we obtain
    \begin{equation}
        \boldsymbol{A}_{i,j} =
        \int_{\manifold}\int_{\complexSpace^{\nfeatures}}
        \metric{\point}{\LC{(f(\data,\point)\log_{\point}(\estimator))}{\basisElement_j}}{\basisElement_i}
        d\data\mu(d\point)
        - 
        \int_{\manifold}\int_{\complexSpace^{\nfeatures}}
        \metric{\point}{f(\data,\point)\LC{\log_{\point}(\estimator)}{\basisElement_j}}{\basisElement_i}
        d\data \mu(d\point).
    \end{equation}
    Assumption~\ref{eq:condition3IBCRB} further yields
    \begin{equation}
        \boldsymbol{A}_{i,j}
        = - \int_{\manifold}\int_{\complexSpace^{\nfeatures}}
        \metric{\point}{f(\data,\point)\LC{\log_{\point}(\estimator)}{\basisElement_j}}{\basisElement_i}
        d\data \mu(d\point)
        = \mathbb{E}_{\data,\point}[\metric{\point}{-\LC{\log_{\point}(\estimator)}{\basisElement_j}}{\basisElement_i}].
    \label{eq:Aij_last}
    \end{equation}
    From there, the end of the proof is very similar to the ones of the (non-Bayesian) intrinsic Cramér-Rao bound that can be found in~\cite{smith05,boumal2013intrinsic}.
    From Equations (38) and (39) in~\cite{smith05}, we have
    \begin{equation}
        \metric{\point}{-\LC{\log_{\point}(\estimator)}{\basisElement_j}}{\basisElement_i}
        =
        \metric{\point}{\basisElement_i}{\basisElement_j}
        - \frac13\metric{\point}{\mathcal{R}(\log_{\point}(\estimator),\basisElement_j)\basisElement_i}{\log_{\point}(\estimator)}
        + \mathcal{O}(\|\log_{\point}(\estimator)\|^3).
    \label{eq:LC_log_approx}
    \end{equation}
    Recall that $\{\basisElement_i\}_{i=1}^{\dimension}$ is an orthonormal basis of $\tangentSpace{\point}$ (i.e., $\metric{\point}{\basisElement_i}{\basisElement_j}=\delta_{i,j}$) and let the operator $R_m$ defined as in~\eqref{eq:curv_term_Bayes}.
    From combining~\eqref{eq:Aij_last} and~\eqref{eq:LC_log_approx} and by neglecting the higher order terms $\mathcal{O}(\|\log_{\point}(\estimator)\|^3$, one gets
    \begin{equation}
        \boldsymbol{A} = \mathbb{E}_{\data,\point}[\errVec\scoreVec^T] = \eye_{\dimension} - \frac13 R_m(\errCovBayes).
    \end{equation}
    Injecting this in~\eqref{eq:inequality_start}, one obtains
    \begin{equation}
        \errCovBayes - \FisherMat_{\textup{B}}^{-1} + \frac13(\FisherMat_{\textup{B}}^{-1}R_m(\errCovBayes) + R_m(\errCovBayes)\FisherMat_{\textup{B}}^{-1}) \succeq \boldsymbol{0}.
    \end{equation}
    which concludes the proof.
\end{proof}

\noindent
Notice that Theorem \ref{thm:main} provides an inequality relating the covariance matrix of the error and the Bayesian Fisher information matrix, as in the classical case of Theorem \ref{thm:vantreesBCRB}. 
However, as for the (non-Bayesian) intrinsic Cramer-Rao bound, the result is obtained according to a chosen error metric, which brings some notable differences regarding the definition of these objects, and the involvement of curvature-related terms within the inequality.
Still, Theorem \ref{thm:main} does not provide a practical way for bounding the expected estimation error because of the terms in $R_m(\errCovBayes)$.
For small errors, the inequality can be simplified to a tractable expression in the following corollary.

\begin{corollary}[First approximation of intrinsic Bayesian Cramér-Rao lower bound] For sufficiently small errors, the inequality in Theorem \ref{thm:main} can be simplified to
\label{coro:mainapprox}
\begin{equation}
        \errCovBayes \succeq \FisherMat_{\textup{B}}^{-1} - \frac13( \FisherMat_{\textup{B}}^{-1} R_m(\FisherMat_{\textup{B}}^{-1}) + R_m(\FisherMat_{\textup{B}}^{-1}) \FisherMat_{\textup{B}}^{-1} ) 
        + \mathcal{O}(\lambda_{\rm max}(\FisherMat_{\textup{B}}^{-1})^{3})
    \label{eq:approx_ineq_Bayes}
    \end{equation}
    where $\lambda_{\rm max}(\FisherMat_{\textup{B}}^{-1})$ is the largest eigenvalue of $\FisherMat_{\textup{B}}^{-1}$, and where the linear operator $R_m(\mathbf{M})$ is defined for any matrix $\mathbf{M}\in\mathbb{C}^{d\times d}$ \cite{boumal2013} as
\begin{equation}
            (R_m(\mathbf{M}))_{i,j} = \sum_{k,l} 
            \langle
            \mathcal{R}(\boldsymbol{\Omega}_k,\boldsymbol{\Omega}_i)\boldsymbol{\Omega}_j,\boldsymbol{\Omega}_l\rangle_{\boldsymbol{\theta}} \mathbf{M}_{k,l} .
\end{equation}
\end{corollary}
\begin{proof}
    Let the operator $\Delta$ be defined as
    \begin{equation}
        \Delta(\errCovBayes) = \frac13(\FisherMat_{\textup{B}}^{-1}R_m(\errCovBayes) + R_m(\errCovBayes)\FisherMat_{\textup{B}}^{-1}) ,
    \end{equation}
    and $\Id$ denotes the identity operator.
    The result in Theorem \ref{thm:main} then reads $(\Id+\Delta)(\errCovBayes)\succeq\FisherMat_{\textup{B}}^{-1}$.
    For small expected errors (referred to as large signal to noise ratio in \cite{smith05}), $\FisherMat_{\textup{B}}^{-1} $ is negligible compared to $ \boldsymbol{I}_d$, so the operator $\Delta$ can be assumed to be a small perturbation of $\Id$.
    Thus, $\Id+\Delta$ is positive definite and its inverse has the first order Taylor expansion $(\Id+\Delta)^{-1}=\Id-\Delta +\Delta^2 -\ldots $
    Applying this approximated inverse operator to both sides of the inequality yields the desired result. 
    The term in $\mathcal{O}(\lambda_{\rm max}(\FisherMat_{\textup{B}}^{-1})^{3})$ corresponds to $\Delta^2$ (that is a linear operator involving an order two of $\FisherMat_{\textup{B}}^{-1}$) applied to $\FisherMat_{\textup{B}}^{-1}$ itself.
\end{proof}

\noindent
As presented in Theorem~\ref{thm:ICRB}, neglecting the curvature terms and taking the trace of the inequality is usually valid in practice.
Notice that neglecting the curvature terms still yields a different result than the standard Cramér-Rao bound because the error is here defined according to any chosen metric. 
Hence we can obtain performance bounds for different Riemannian distances (rather than the mean squared error).
\begin{corollary}[Second Approximation of intrinsic Bayesian Cramér-Rao lower bound]
\label{coro:asymp_approx}
Neglecting curvature terms in \eqref{eq:approx_ineq_Bayes}, the expected estimation error is bounded as
\begin{equation}
    \mathbb{E}_{\data,\point}[d_R^2(\estimator,\point)] \geq \tr((\mathbb{E}_{\point}[\FisherMat_{\point}] + \FisherMat_{\textup{prior}})^{-1}),
\end{equation}
where $\FisherMat_{\point}$ is defined in~\eqref{eq:Def_Ftheta} and $\FisherMat_{\textup{prior}}$ is defined as
\begin{equation}
    (\FisherMat_{\textup{prior}})_{i,j}
    = \mathbb{E}_{\point}[\Diff\log f_{\textup{prior}}(\point)[\basisElement_i] \cdot \Diff\log f_{\textup{prior}}(\point)[\basisElement_j]]
    = \mathbb{E}_{\point}[-\Diff^2\log f_{\textup{prior}}(\point)[\basisElement_i,\basisElement_j]].
\end{equation}
\end{corollary}
\begin{proof}
We start by removing the curvature terms in Theorem \ref{thm:main}, which reads $\errCovBayes \succeq \FisherMat_{\textup{B}}^{-1}$.
Next, we take the trace of this inequality which yields
\begin{equation}
    \mathbb{E}_{\data,\point}[d_R^2(\estimator,\point)] \geq \tr(\FisherMat^{-1}_{\textup{B}}).
\end{equation}
As for the Euclidean Bayesian case (cf. Equation~\eqref{eq:FisherMat_Bayes_Eucl}), since the joint log-likelihood can be expressed as $\log f(\data,\point) = \log f(\data|\point)+\log f_{\textup{prior}}(\point)$, one can decompose the Bayesian Fisher information matrix as $\FisherMat_{\textup{B}} = \mathbb{E}_{\point}[\FisherMat_{\point}] + \FisherMat_{\textup{prior}}$, which concludes the proof.
\end{proof}

%% COVARIANCE IBCRLB
\section{Intrinsic Bayesian Cramér-Rao bounds for covariance matrix estimation}
\label{sec:covariance crb}

We consider the fundamental problem of covariance matrix estimation when the data is assumed to follow a centered Gaussian distribution.
In the corresponding Bayesian setup, using the conjugate prior consists in assuming that the underlying covariance matrix is drawn from an inverse Wishart distribution.
Though this prior distribution has been extensively used in practical estimation problems~\cite{barnard2000modeling, BOURIGA2013795, Fraley2007, Pourahmadi, guerci2006knowledge, besson07, besson2008bounds, Bandiera2010, Bandiera2010Tyler, de2010knowledge, bidon2011knowledge, wang2017cfar}, there is, to the best of our knowledge, no derivation of a Bayesian performance bound for estimation in this setup (even in the classical framework of Section \ref{sec:basic inequalities}).
Following from Corollary \ref{coro:asymp_approx}, we obtain these bounds for two performance metrics: the Euclidean one (that yields to a standard performance bound on the mean squared error), and the affine invariant metric for Hermitian positive definite matrices (that yields a bound on the a distance more naturally suited to the space of covariance matrices).
Hence, the two proposed Theorems extend the results of \cite{smith05} on covariance matrix estimation to a Bayesian context. 
%

%%%%%%%%%%%%%%%%%%%%%%%%%%%%%%%%%%%%%%%%%%%%%%%%%%%%%%%%%%%%%%%%%%%%%%%%%%%%%%%%%%%%%%%%%%%%%%%
%%%%%%%%%%%%%%%%%%%%%%%%%%%%%%%%%%%%%%%%%%%%%%%%%%%%%%%%%%%%%%%%%%%%%%%%%%%%%%%%%%%%%%%%%%%%%%%

\subsection{Data model and MAP (Maximum A Posteriori) and MMSE (Minimum Mean Square Error) estimators}

We first present the data model, as well as the associated maximum a posteriori and minimum mean square error estimators: 
the sample set $\{\mathbf{y}_k\}_{k=1}^n\in\mathbb{C}^p$ is assumed to be i.i.d. and drawn from a Complex Gaussian distribution, denoted $\mathbf{y}\sim \mathcal{CN}(\mathbf{0},\boldsymbol{\Sigma})$, with (unknown) covariance matrix.
This covariance matrix belongs to the manifold of Hermitian positive definite matrices, i.e., $\boldsymbol{\Sigma}\in\mathcal{H}_p^{++}$, where
    \begin{equation}
    \mathcal{H}_p^{++}    = 
    \left[ 
    \boldsymbol{\Sigma}
    \in 
    \mathcal{H}_p  :~
    ~\forall~\mathbf{x}\in\mathbb{C}^{p} \backslash \{ \mathbf{0} \},
    ~\mathbf{x}^H\boldsymbol{\Sigma} \mathbf{x}> 0
    \right]
    ,
\label{eq:hpd_space}
\end{equation}
and where $\mathcal{H}_p $ denotes the set of $p\times p$ Hermitian matrices. 
The log-likelihood for this sample set is expressed as:
\begin{equation}
\log f(\{ \mathbf{y}_k\}_{k=1}^n|\boldsymbol{\Sigma}) \propto -n \log |\boldsymbol{\Sigma}| - \sum_{k=1}^n \tr(\boldsymbol{\Sigma}^{-1}\mathbf{y}_k\mathbf{y}_k^H).
\end{equation}
Moreover we assume that the covariance matrix is drawn from an inverse Wishart distribution, denoted $\boldsymbol{\Sigma}\sim \mathcal{IW}((\nu-p)\boldsymbol{\Sigma}_0,\nu)$, with center $\boldsymbol{\Sigma}_0$ and $\nu$ degrees of freedom (both known).
Notice that the choice of scaling in this definition is set so that we have a consistent mean $\mathbb{E}_{\boldsymbol{\Sigma}} [\boldsymbol{\Sigma}] =\boldsymbol{\Sigma}_0$ for any value of $\nu\leq p+1$.
The corresponding probability density function with respect to the Lebesgue measure ($\mu(d\mathbf{\Sigma})=d\mathbf{\Sigma}$) restricted to $\mathcal{H}_p^{++}$:
\begin{equation}
\log f_{\mathcal{IW}}(\boldsymbol{\Sigma}|\boldsymbol{\Sigma}_0,\nu) \propto -(\nu+p)\log|\boldsymbol{\Sigma}|-\tr((\nu-p)\boldsymbol{\Sigma}^{-1}\boldsymbol{\Sigma}_0).
\end{equation}
The choice of this conjugate prior of the Gaussian distribution
to model prior knowledge about the covariance matrix has been leveraged in numerous works\footnote{A notable mention is \cite{besson2008bounds} that derives Cramér-Rao bounds for the different problem of estimating the deterministic parameter $\boldsymbol{\Sigma}_0$ from heterogeneous samples $\mathbf{y}_k\sim\mathcal{CN}(\mathbf{0},\boldsymbol{\Sigma}_k)$ with $\boldsymbol{\Sigma}_k \sim \boldsymbol{\Sigma}\sim \mathcal{IW}((\nu-p)\boldsymbol{\Sigma}_0,\nu)$.
The considered Bayesian scenario of estimating the random parameter $\boldsymbol{\Sigma}$ with prior information about its distribution is more classical, and used in most of the other mentioned references. For this setup, performance bounds were, to the best of our knowledge, not investigated.} \cite{barnard2000modeling, BOURIGA2013795, Fraley2007, Pourahmadi, guerci2006knowledge, besson07, besson2008bounds, Bandiera2010, Bandiera2010Tyler, de2010knowledge, bidon2011knowledge, wang2017cfar}.
Given those assumptions, the maximum a posteriori (MAP) estimator of $\boldsymbol{\Sigma}$ is defined as
\begin{equation}
\hat{\boldsymbol{\Sigma}}_{MAP-\mathcal{IW}} = \mbox{argmax}_{\boldsymbol{\Sigma}} \quad f(\{ \mathbf{y}_k\}_{k=1}^n|\boldsymbol{\Sigma})  f_{\mathcal{IW}}(\boldsymbol{\Sigma}|\boldsymbol{\Sigma}_0,\nu).
\end{equation}
The log of the posterior distribution is
\begin{equation}
\log f_p \propto -(\nu+p+n) \log|\boldsymbol{\Sigma}|-(\nu-p)\tr(\boldsymbol{\Sigma}^{-1}\boldsymbol{\Sigma}_0)-\sum_{k=1}^n \tr(\boldsymbol{\Sigma}^{-1}\mathbf{y}_k\mathbf{y}_k^H),
\end{equation}
from which is is easy to obtain a closed form expression for the MAP estimator:
\begin{equation}
\hat{\boldsymbol{\Sigma}}_{MAP-\mathcal{IW}} = \frac{1}{\nu+n+p}((\nu-p)\boldsymbol{\Sigma}_0+\sum_{k=1}^n\mathbf{y}_k\mathbf{y}_k^H).
\label{eq:comp_map_iw}
\end{equation}

The MMSE estimator is given by the posterior mean:
\begin{equation}
    \mathbb{E}(\boldsymbol{\Sigma} | \{\mathbf{y}_k\}_{k=1}^n) = \int_{\mathcal{H}^{++}} \boldsymbol{\Sigma} f(\boldsymbol{\Sigma} | \{\mathbf{y}_k\}_{k=1}^n) d\boldsymbol{\Sigma}.
    \label{eq:postmean}
\end{equation}
From \cite{Bandiera2010}, we have:
\begin{equation}
    \boldsymbol{\Sigma} \sim \mathcal{IW}\left((\nu-p)\boldsymbol{\Sigma}_0+\sum_{k=1}^n \mathbf{y}_k\mathbf{y}_k^H,\nu+n\right).
\end{equation}
Therefore the computation of the mean \eqref{eq:postmean} is direct and gives us:
\begin{equation}
    \hat{\boldsymbol{\Sigma}}_{MMSE-\mathcal{IW}} = \frac{1}{\nu+n-p} \left( (\nu-p) \boldsymbol{\Sigma}_0 + \sum_{k=1}^n \mathbf{y}_k \mathbf{y}_k^{H} \right).
    \label{eq:MMSE}
\end{equation}

Finally, we emphasize that $ \mathcal{H}_p^{++} $ in \eqref{eq:hpd_space} is a smooth manifold, as it is an open of the linear space $\mathcal{H}_p $.
Its tangent space at each point $\boldsymbol{\Sigma}$ is identifiable to $T_{\boldsymbol{\Sigma}}\mathcal{H}_p^{++} =\mathcal{H}_p $.
Endowing this tangent space with a Riemannian metric yields a Riemannian manifold and its corresponding geometry (cf. \cite{bouchard2023fisher} for a more detailed introduction on this topic).
In the following, we study the intrinsic Bayesian Cramér-Rao bound for two of these metrics: the Euclidean metric, and the affine invariant metric on $\mathcal{H}_p^{++}$.

\subsection{Performance bound based on the Euclidean metric}
\label{subsec:euclidean_bound}

When endowing $\mathcal{H}_p^{++}$ with the Euclidean metric\footnote{When dealing with complex-valued matrices, the real part ensures that the metric defines a proper inner product on $\mathcal{H}_p^{++}$. 
As in most works (e.g., \cite{smith05}), we will however keep this implicit, and omit the notation $\mathfrak{Re}$ to lighten the exposition (in particular in the derivations of the appendix).}
\begin{equation}
    \langle 
    \mathbf{\Omega}_i , \mathbf{\Omega}_j
    \rangle_{\boldsymbol{\Sigma}}^{\mathcal{E}}
    = \mathfrak{Re} \{ {\tr}(\mathbf{\Omega}_i \mathbf{\Omega}_j ) \},~\forall \mathbf{\Omega}_i, \mathbf{\Omega}_j 
    \in T_{\boldsymbol{\Sigma}}\mathcal{H}_p^{++}
    ,
\label{eq:euclidean_metric}
\end{equation}
With this metric, the error vector constructed from the Riemannian logarithm is simply the standard error vector, i.e.:
\begin{equation}
 \log^{\mathcal{E}}_{\boldsymbol{\Sigma}}(\hat{\boldsymbol{\Sigma}})=\hat{\boldsymbol{\Sigma}}-\boldsymbol{\Sigma},
\end{equation}
and the corresponding error measure is the Euclidean distance
\begin{equation}
       d_E^2(\boldsymbol{\Sigma},\hat{\boldsymbol{\Sigma}})=\|\hat{\boldsymbol{\Sigma}}-\boldsymbol{\Sigma}\|_F^2.
       \label{eq:euclidean_dist}
\end{equation}
For the upcoming derivations, we first need to define a basis of $T_{\boldsymbol{\Sigma}} \mathcal{H}_{p}^{++}$ that is orthonormal with respect to the Euclidean metric \eqref{eq:euclidean_metric}. 
In practice, we will use the canonical basis $\{ \mathbf{\Omega}_i^{E} \}_{i\in [\![ 1,p^2]\!]}$ that is indexed over $i\in [\![ 1, p^2]\!]$ as follows:
\begin{itemize}
\item 
For indexes $i \in [\![1,p ]\!]$, the basis elements $\mathbf{\Omega}_i^{E}\in T_{\boldsymbol{\Sigma}} \mathcal{H}_{p}^{++} $ refer to the matrices denoted $\mathbf{\Omega}_{ii}^{E}$ which are $p$ by $p$ symmetric matrices whose $i$th diagonal element is one, and zeros elsewhere
\item 
For indexes $i \in [\![p+1,p(p+1)/2 ]\!]$, the basis elements $\mathbf{\Omega}_i^{E}\in T_{\boldsymbol{\Sigma}} \mathcal{H}_{p}^{++}$ refer to the matrices denoted $\mathbf{\Omega}_{mn}^{E}$ which are $p$ by $p$ symmetric matrices whose $mn$th and $nm$th elements are both $2^{-1/2}$, and zeros elsewhere. The mapping between the index $i$ and $(m,n) \in [((1,2),(1,3)\ldots(1,p),(2,3),\ldots\ldots(p-1,p)]$ is left implicit.
\item 
For indexes $i \in [\![ p(p+1)/2+1 , p^2 ]\!]$, the basis elements $\mathbf{\Omega}_i^{E}\in T_{\boldsymbol{\Sigma}} \mathcal{H}_{p}^{++}$ refer to the matrices denoted $\mathbf{\Omega}_{mn}^{h-{E}}$ which are $p$ by $p$ Hermitian matrices whose $mn$th element is $2^{-1/2}\sqrt{-1}$, and $nm$th element is 
$-2^{-1/2}\sqrt{-1}$, and zeros elsewhere.
Again, the second mapping between the index $i$ and $(m,n) \in [((1,2),(1,3)\ldots(1,p),(2,3),\ldots\ldots(p-1,p)]$ is left implicit.
\end{itemize}

For the Euclidean distance \eqref{eq:euclidean_dist}, we obtain the following bound:

\begin{theorem}[Bayesian Euclidean Cramér-Rao bound for covariance matrix estimation] \label{thm:eucl_bcrb_iw}
Let $\{\mathbf{y}_k\}_{k=1}^n\in\mathbb{C}^p$ be i.i.d. as $\mathbf{y}\sim \mathcal{CN}(\mathbf{0},\boldsymbol{\Sigma})$, with $\boldsymbol{\Sigma}\sim\mathcal{IW}((\nu-p)\boldsymbol{\Sigma}_0,\nu)$.
Let  $\hat{\boldsymbol{\Sigma}}$ be an estimator of $\boldsymbol{\Sigma}$ built from $\{\mathbf{y}_k\}_{k=1}^n\in\mathbb{C}^p$, then
\begin{equation}
    \mathbb{E}_{\data,\boldsymbol{\Sigma}}\left[d_E^2(\boldsymbol{\Sigma},\hat{\boldsymbol{\Sigma}})\right] \geq \tr(\left.\FisherMat^\mathcal{E}_{\mathcal{IW}}\right.^{-1}),
\end{equation}
where $\forall (i,j)\in [\![1,p^2 ]\!]^2$
\begin{equation}
    \left(\FisherMat^\mathcal{E}_{\mathcal{IW}}\right)_{i,j} = \frac{n\nu^2}{(\nu-p)^2} \tr(\boldsymbol{\Sigma}_0^{-1}\boldsymbol{\Omega}_i^E\boldsymbol{\Sigma}_0^{-1}\boldsymbol{\Omega}^E_j)+ \frac{n\nu}{(\nu-p)^2} \tr(\boldsymbol{\Sigma}_0^{-1}\boldsymbol{\Omega}_i^E)\tr(\boldsymbol{\Sigma}_0^{-1}\boldsymbol{\Omega}^E_j) + \FisherMat^{\mathcal{E}}_{\textup{prior}}(i,j)
    \label{eq:fim_euclidean}
\end{equation}
and
\begin{equation}
    \left(\FisherMat^{\mathcal{E}}_{\textup{prior}}\right)_{i,j}  = \alpha \tr(\boldsymbol{\Sigma}_0^{-1}\boldsymbol{\Omega}_i^E\boldsymbol{\Sigma}_0^{-1}\boldsymbol{\Omega}^E_j)  + \beta \tr(\boldsymbol{\Sigma}_0^{-1}\boldsymbol{\Omega}_i^E)\tr(\boldsymbol{\Sigma}_0^{-1}\boldsymbol{\Omega}^E_j),
\end{equation}
where $\alpha = \frac{\nu^3+p\nu^2+2\nu}{(\nu-p)^2} $ and $\beta = \frac{3\nu^2-p\nu}{(\nu-p)^2}$. 
\end{theorem}
\begin{proof}
    cf Section \ref{subsec:proof_outline} and appendix \ref{appen:proof_euclidian}.
\end{proof}

\subsection{Performance bound based on the affine invariant metric}

The affine invariant metric on $\mathcal{H}_p^{++}$ is defined as \cite{bhatia2009positive,skovgaard1984riemannian} (see footnote 3 concerning the real part operator $\mathfrak{Re}$):
\begin{equation}
    \langle 
    \mathbf{\Omega}_i , \mathbf{\Omega}_j
    \rangle_{\boldsymbol{\Sigma}}^{\rm AI}
    = \mathfrak{Re} \{ {\tr}( \boldsymbol{\Sigma}^{-1} \mathbf{\Omega}_i  \boldsymbol{\Sigma}^{-1}  \mathbf{\Omega}_j ) \},~\forall \mathbf{\Omega}_i, \mathbf{\Omega}_j 
    \in T_{\boldsymbol{\Sigma}}\mathcal{H}_p^{++}
    .
    \label{eq:AI metric}
\end{equation}
Among many other interesting invariance properties, this metric is quadratically dependent on the point $\boldsymbol{\Sigma}^{-1}$.
This makes the norm of tangent vectors (cf. Figure \ref{fig:manifold}) tends to infinity when the point $\boldsymbol{\Sigma}$ goes to the boundaries of  $\mathcal{H}_p^{++}$ (when any number of its eigenvalues tend to $0$).
This Riemannian metric therefore allows us to perceive the boundary of the manifold as being infinitely far, which is more in accordance with the actual geometry of this space than when using the Euclidean metric.
The affine invariant metric also appears as the Fisher information metric of the Gaussian model, and the study of the corresponding geometry may reveal unexpected properties of related estimation problems \cite{smith05}.
When endowing $\mathcal{H}_p^{++}$ with the metric \eqref{eq:AI metric}, the error vector constructed from the Riemannian logarithm is defined as
\begin{equation}
 \log_{{\boldsymbol{\Sigma}}}^{\rm AI}(\hat{\boldsymbol{\Sigma}})
    \; = \; {\boldsymbol{\Sigma}} \textup{logm}({\boldsymbol{\Sigma}}^{-1}\hat{\boldsymbol{\Sigma}}),
\end{equation}
where $\textup{logm}$ denotes the standard matrix logarithm.
The corresponding error measure is the so-called natural Riemannian distance on $\mathcal{H}_p^{++}$:
\begin{equation}
       d_{\rm AI}^2({\boldsymbol{\Sigma}},\hat{\boldsymbol{\Sigma}})  = \| \textup{logm}({\boldsymbol{\Sigma}}^{-1}\hat{\boldsymbol{\Sigma}})\|_2^2.
       \label{eq:affine_distance}
\end{equation}
For the natural Riemannian distance \eqref{eq:affine_distance}, we obtain the following bound:

\begin{theorem}[Intrinsic Bayesian Cramér-Rao bound for covariance matrix estimation]  \label{thm:riem_bcrb_iw}
Let $\{\mathbf{y}_k\}_{k=1}^n\in\mathbb{C}^p$ be i.i.d. as $\mathbf{y}\sim \mathcal{CN}(\mathbf{0},\boldsymbol{\Sigma})$, with $\boldsymbol{\Sigma}\sim\mathcal{IW}((\nu-p)\boldsymbol{\Sigma}_0,\nu)$.
Let  $\hat{\boldsymbol{\Sigma}}$ be an estimator of $\boldsymbol{\Sigma}$ built from $\{\mathbf{y}_k\}_{k=1}^n\in\mathbb{C}^p$, then the application of Corollary \ref{coro:asymp_approx} yields
\begin{equation}
    \mathbb{E}_{\data,\boldsymbol{\Sigma}}\left[d_{\rm AI}^2({\boldsymbol{\Sigma}},\hat{\boldsymbol{\Sigma}})\right] \geq \tr(\left.\FisherMat^{AI}_{\mathcal{IW}}\right.^{-1}),
\end{equation}
where
\begin{equation}
    \FisherMat^{AI}_{\mathcal{IW}} = n\boldsymbol{I}_{p^2} + \FisherMat^{AI}_{\textup{prior}}
    \label{eq:FIM_AI}
\end{equation}
and $\forall (i,j)\in [\![1,p^2 ]\!]^2$
\begin{equation}
    \begin{array}{ll}
      \left(\FisherMat^{AI}_{\textup{prior}}\right)_{i,j}  = (\nu+p)^2 \delta_{i,j} - (\nu+p)\left( \left(\boldsymbol{f}^{AI}_{\textup{prior}}\right)_{i}+\left(\boldsymbol{f}^{AI}_{\textup{prior}}\right)_{j} \right) + \left(\bar{\FisherMat}^{AI}_{\textup{prior}}\right)_{i,j}    &  \mbox{~if~} i,j \leq p \\
       \left(\FisherMat^{AI}_{\textup{prior}}\right)_{i,j}  = \left(\bar{\FisherMat}^{AI}_{\textup{prior}}\right)_{i,j}  & \mbox{~else~}
    \end{array},
\end{equation}
with 
\begin{equation}
\left(\boldsymbol{f}^{AI}_{\textup{prior}}\right)_{i}  = \left[
\begin{array}{l}
\nu+p-2i+1 \mbox{~if~} i\leq p \\
0 \mbox{~if~} i> p
\end{array} \right.
\end{equation}
and
\begin{equation}
\left(\bar{\FisherMat}^{AI}_{\textup{prior}}\right)_{i,j}  = \left[
\begin{array}{l l}
\nu-i+1+(\nu+p-2i+1)^2 
& \text{if~}  (i,j)\in [\![1,p]\!]^2 \text{~and~} i=j\\
%\mathbf{\Omega}_{i}^{E}=\mathbf{\Omega}_{j}^{E}:=\mathbf{\Omega}_{ii}^{E} \\
(\nu+p-2i+1)(\nu+p-2j +1) 
& \text{if~} \text{~if~}  (i,j)\in [\![1,p]\!]^2 \text{~and~} i\neq j\\
% \nu+p-2m+1 
%& \text{if~} i \in [\![p+1,p(p+1)/2 ]\!]
%\text{~and~}~i=j
%\boldsymbol{\Omega}_i^E =  \boldsymbol{\Omega}_j^E = \boldsymbol{\Omega}_{mn}^E \\
\nu+p-2m+1 
&\text{if~} 
i >p \mbox{~and~}~i=j \\
%\boldsymbol{\Omega}_i^E =  \boldsymbol{\Omega}_j^E = \boldsymbol{\Omega}_{mn}^{h-E} \\
0 
&\text{otherwise~}
\end{array} \right.
\end{equation}
and where $m$ stands for the column index of non-zero element in the element of the basis $\{\boldsymbol{\Omega}_i^E\}_{i=1}^{p^2}$ when identifying $ \boldsymbol{\Omega}_i^E = \boldsymbol{\Omega}_{mn}^E$ if $i \in [\![p+1,p(p+1)/2 ]\!]$, or $ \boldsymbol{\Omega}_i^{E} = \boldsymbol{\Omega}_{mn}^{h-E}$ if $i \in [\![ p(p+1)/2+1 , p^2 ]\!]$.
\end{theorem}
\begin{proof}
    cf section \ref{subsec:proof_outline} and Appendix \ref{appen:proof_affine}.
\end{proof}

\begin{remark}
The main justification for omitting the non-zero curvature terms in the case of the affine invariant metric (i.e. applying Corollary \ref{coro:asymp_approx} rather than Corollary \ref{coro:mainapprox}) is that these have been shown to be negligible in the case of deterministic covariance matrix estimation in \cite{smith05} (which was also observed for other distributions \cite{Breloy2019}).
Ultimately, our simulations also validate this choice because the performance of the estimators is well described by the bound without needing these additional terms.
\end{remark}

To conclude on this result, we notice that, contrarily to the Euclidean case of Theorem \ref{thm:eucl_bcrb_iw}, the intrinsic bound on the natural Riemannian distance in Theorem \ref{thm:riem_bcrb_iw} does not depend on the hyperparameter $\boldsymbol{\Sigma}_0$, but only on the dimensions $n$, $p$, and the degrees of freedom $\nu$.
This interesting property offers an easier interpretation of the bound, but will also have a more fundamental impact (the possibility of statistical efficiency regarding the chosen metric), as illustrated in the next section.

\subsection{Proofs outline}
\label{subsec:proof_outline}

Independently of the chosen metric, we consider that we have an orthonormal basis of the tangent space that is denoted $\{\mathbf{\Omega}_i\}_{i=1}^{p^2}$ (the canonical one for the Euclidean metric and a more complex for the affine invariant metric which is developed in appendix \ref{appen:proof_affine}).
We then have to compute the matrix\footnote{From now, all expectations are only taken over $\boldsymbol{\Sigma}$, but denoted $\mathbb{E}$ to lighten the exposition.}
\begin{equation}
\FisherMat_{\mathcal{IW}} = \mathbb{E} \left[\FisherMat_{\boldsymbol{\Sigma}}(\boldsymbol{\Sigma}) \right] +  \FisherMat_{\textup{prior}} ,
\end{equation}
where 
\begin{equation}
    (\FisherMat_{\boldsymbol{\Sigma}}(\boldsymbol{\Sigma}))_{i,j} = n\tr\left(\boldsymbol{\Sigma}^{-1} \boldsymbol{\Omega}_i\boldsymbol{\Sigma}^{-1}\boldsymbol{\Omega}_j\right)
\end{equation}
is the Fisher information matrix of the Gaussian model (cf. \cite{smith05, Breloy2019, bouchard2023fisher} for its derivation with matching notations).
The second term $\FisherMat_{\textup{prior}}$ is defined as
\begin{equation}
(\FisherMat_{\textup{prior}})_{i,j} = \mathbb{E} \left[ \Diff_{\boldsymbol{\Sigma}} \log f_{\mathcal{IW}}(\boldsymbol{\Sigma})[\boldsymbol{\Omega}_i] .\Diff_{\boldsymbol{\Sigma}} \log f_{\mathcal{IW}}(\boldsymbol{\Sigma})[\boldsymbol{\Omega}_j] \right].
\end{equation}
Given the directional derivatives
\begin{equation}
\begin{array}{l}
\Diff_{\boldsymbol{\Sigma}} \left( \log |\boldsymbol{\Sigma}| \right) [\boldsymbol{\xi}]=\tr(\boldsymbol{\Sigma}^{-1}\boldsymbol{\xi})
\\
\Diff_{\boldsymbol{\Sigma}}
\left( \tr(\boldsymbol{\Sigma}^{-1}\boldsymbol{\Sigma}_0) \right)
[\boldsymbol{\xi}]=-\tr(\boldsymbol{\Sigma}^{-1}\boldsymbol{\xi}\boldsymbol{\Sigma}^{-1}\boldsymbol{\Sigma}_0).
\end{array}
\end{equation}
We have: 
\begin{equation}
\Diff_{\boldsymbol{\Sigma}} \log f_{\mathcal{IW}}(\boldsymbol{\Sigma})[\boldsymbol{\Omega}_i]=-(\nu+p)\tr(\boldsymbol{\Sigma}^{-1}\boldsymbol{\Omega}_i)+(\nu-p)\tr(\boldsymbol{\Sigma}^{-1}\boldsymbol{\Omega}_i\boldsymbol{\Sigma}^{-1}\boldsymbol{\Sigma}_0).
\end{equation}
Hence, we have
\begin{equation}
\begin{array}{lll}
(\FisherMat_{\textup{prior}})_{i,j}& = &  (\nu+p)^2 \mathbb{E} \left[ \tr(\boldsymbol{\Sigma}^{-1}\boldsymbol{\Omega}_i)\tr(\boldsymbol{\Sigma}^{-1}\boldsymbol{\Omega}_j)\right]   \\
& & +(\nu-p)^2 \mathbb{E}\left[\tr(\boldsymbol{\Sigma}^{-1}\boldsymbol{\Omega}_i\boldsymbol{\Sigma}^{-1}\boldsymbol{\Sigma}_0)\tr(\boldsymbol{\Sigma}^{-1}\boldsymbol{\Omega}_j\boldsymbol{\Sigma}^{-1}\boldsymbol{\Sigma}_0)\right] \\
& &   -(\nu^2-p^2)\mathbb{E}\left[\tr(\boldsymbol{\Sigma}^{-1}\boldsymbol{\Omega}_i)\tr(\boldsymbol{\Sigma}^{-1}\boldsymbol{\Omega}_j\boldsymbol{\Sigma}^{-1}\boldsymbol{\Sigma}_0)\right]   \\
& &    -(\nu^2-p^2)\mathbb{E}\left[\tr(\boldsymbol{\Sigma}^{-1}\boldsymbol{\Omega}_j)\tr(\boldsymbol{\Sigma}^{-1}\boldsymbol{\Omega}_i\boldsymbol{\Sigma}^{-1}\boldsymbol{\Sigma}_0)\right] . 
  \end{array}
  \label{eq:comp_fiw_1a}
\end{equation}
Hence, the computation of $\FisherMat_{\textup{prior}}$ and $\mathbb{E} \left[\FisherMat_{\boldsymbol{\Sigma}}(\boldsymbol{\Sigma}) \right]$ involves the calculus of four core expectations:
\begin{equation}
    \begin{array}{l}
    T_1(i,j) = \mathbb{E}\left[\tr(\boldsymbol{\Sigma}^{-1}\boldsymbol{\Omega}_i)\tr(\boldsymbol{\Sigma}^{-1}\boldsymbol{\Omega}_j)\right]
    \\
    T_2(i,j) = \mathbb{E}\left[\tr(\boldsymbol{\Sigma}^{-1}\boldsymbol{\Omega}_i\boldsymbol{\Sigma}^{-1}\boldsymbol{\Sigma}_0)\tr(\boldsymbol{\Sigma}^{-1}\boldsymbol{\Omega}_j\boldsymbol{\Sigma}^{-1}\boldsymbol{\Sigma}_0)\right]
    \\
    T_3(i,j) =\mathbb{E}\left[\tr(\boldsymbol{\Sigma}^{-1}\boldsymbol{\Omega}_j)\tr(\boldsymbol{\Sigma}^{-1}\boldsymbol{\Omega}_i\boldsymbol{\Sigma}^{-1}\boldsymbol{\Sigma}_0)\right]
    \\
    T_4(i,j) =\mathbb{E}\left[\tr\left(\boldsymbol{\Sigma}^{-1} \boldsymbol{\Omega}_i\boldsymbol{\Sigma}^{-1}\boldsymbol{\Omega}_j\right)\right]
    \end{array} .
    \label{eq:all_Ts}
\end{equation}
At this point, the proof of the two Theorems differ because we need to define a basis of the tangent space $\{\mathbf{\Omega}_i\}_{i=1}^{p^2}$ that is orthogonal with respect to the chosen error metric:
\begin{itemize}
    \item 
    For the Euclidean metric, we can use the standard canonical basis of $\mathcal{H}_p$ presented in section \ref{subsec:euclidean_bound}.
    The computation of the four expectations then requires to derive new results about traces of functions of inverse Wishart matrices. Full derivations are provided in Appendix \ref{appen:proof_euclidian}.
    
    \item
    The affine invariant metric requires a more careful construction of the basis thanks to the Bartlett decomposition, in order for the expectations to be actually tractable.
    The detailed derivation are provided in Appendix \ref{appen:proof_affine}.
    
\end{itemize}
These tedious derivations being achieved, we have all the necessary objects to obtain the final inequalities: the results are obtained by applying Corollary \ref{coro:asymp_approx} and the relation \eqref{eq:link_dist_log}.
As in \cite{smith05} the curvature terms are neglected in the case of the affine invariant metric.

\section{Numerical experiments}
	\label{sec:num_exp}
	
\subsection{Parameters and method}

We consider a Toeplitz matrix $(\mathbf{\Sigma}_0)_{i,j} = \rho^{|i-j|}$ where $\rho = 0.5$. The data size is $p=5$.
The degrees of freedom for the $\mathcal{IW}$ prior is denoted $\nu$ and will take two values $\nu=40,100$. 
The number of trials for estimating the MSE (Mean Square Error) is equal to 1000. 
We sample the data by using $\mathbf{\Sigma}_{0}$ and $\mathbf{\Sigma}_{\mathcal{IW}}$. 
We compute the SCM, the MAP \eqref{eq:comp_map_iw} and the MMSE estimators \eqref{eq:MMSE}. 
And finally, we evaluate the MSE either by using the Euclidean distance $d_{\mathcal{E}}^2$ for the Euclidean inequality or the natural distance $d_{\mathcal{AI}}^2$ for the intrinsic inequality. 
The Euclidean and intrinsic Cramér-Rao bounds are computed by using the formulas from theorems \ref{thm:eucl_bcrb_iw} and \ref{thm:riem_bcrb_iw}. The code is available at the following address \url{https://github.com/flbouchard/intrinsic_Bayesian_CRB}.

\subsection{Results}

Figures \ref{fig:euclidCRB} and \ref{fig:NaturCRB} show the MSE computed either with the Euclidean distance or the natural distance and the corresponding Cramér-Rao bounds (CRB).
We retrieve the classical result in the Euclidean case where the SCM is proven to be efficiency. 
On the opposite, this efficiency is lost when we consider the natural properties of a covariance matrix. This result is also known and firstly reported in \cite{smith05} and shows the interest of the intrinsic bound in demonstrating properties of an estimator that the classical Euclidean bound fails to exhibit.

\begin{figure}[h]
	\centering
	\begin{subfigure}[b]{0.49\textwidth}
	\centering
	\scalebox{0.78}{% This file was created with tikzplotlib v0.10.1.
\begin{tikzpicture}

\definecolor{darkgray176}{RGB}{176,176,176}
\definecolor{darkorange25512714}{RGB}{255,127,14}
\definecolor{lightgray204}{RGB}{204,204,204}
\definecolor{steelblue31119180}{RGB}{31,119,180}

\begin{axis}[
legend cell align={left},
legend style={fill opacity=0.8, draw opacity=1, text opacity=1, draw=lightgray204},
log basis x={10},
log basis y={10},
tick align=outside,
tick pos=left,
title={Parameter: \(\displaystyle p=5\)},
x grid style={darkgray176},
xlabel={\(\displaystyle n\)},
xmin=7.49897501664337, xmax=4216.57625606459,
xmode=log,
xtick style={color=black},
y grid style={darkgray176},
ylabel={MSE \(\displaystyle \left( \mathbf{\Sigma}, \hat{\mathbf{\Sigma}} \right)\)},
ymin=0.00589410899224351, ymax=3.42156640125114,
ymode=log,
ytick style={color=black}
]
\addplot [semithick, steelblue31119180, mark=*, mark size=3, mark options={solid}]
table {%
10 2.56210775121505
15 1.66549272460532
23 1.06863454049767
34 0.737077951078505
52 0.48004893348822
79 0.316587255913915
119 0.207251566397025
179 0.13753299030681
270 0.0931539045316494
408 0.062037215156532
614 0.0412561824327547
925 0.0266135011207831
1394 0.0177954196059272
2099 0.0120295909615954
3162 0.00787128694474641
};
\addlegendentry{MSE SCM }
\addplot [semithick, darkorange25512714, mark=asterisk, mark size=3, mark options={solid}]
table {%
10 2.5
15 1.66666666666666
23 1.08695652173913
34 0.735294117647058
52 0.48076923076923
79 0.316455696202531
119 0.210084033613445
179 0.139664804469274
270 0.0925925925925924
408 0.0612745098039215
614 0.0407166123778501
925 0.027027027027027
1394 0.0179340028694404
2099 0.0119104335397808
3162 0.00790638836179632
};
\addlegendentry{CRB}
\end{axis}

\end{tikzpicture}}
    \caption{Euclidean Cramér-Rao bound}
    \label{fig:euclidCRB}  
    \end{subfigure}
    \begin{subfigure}[b]{0.49\textwidth}
	\centering
	\scalebox{0.78}{% This file was created with tikzplotlib v0.10.1.
\begin{tikzpicture}

\definecolor{darkgray176}{RGB}{176,176,176}
\definecolor{darkorange25512714}{RGB}{255,127,14}
\definecolor{lightgray204}{RGB}{204,204,204}
\definecolor{steelblue31119180}{RGB}{31,119,180}

\begin{axis}[
legend cell align={left},
legend style={fill opacity=0.8, draw opacity=1, text opacity=1, draw=lightgray204},
log basis x={10},
log basis y={10},
tick align=outside,
tick pos=left,
title={Parameter: \(\displaystyle p=5\)},
x grid style={darkgray176},
xlabel={\(\displaystyle n\)},
xmin=7.49897501664337, xmax=4216.57625606459,
xmode=log,
xtick style={color=black},
y grid style={darkgray176},
ymin=0.00578995788766281, ymax=5.48700863128032,
ymode=log,
ytick style={color=black}
]
\addplot [semithick, steelblue31119180, mark=*, mark size=3, mark options={solid}]
table {%
10 4.01821254542287
15 2.23972211241336
23 1.29822902723328
34 0.840118669755483
52 0.530336878371893
79 0.334954737315194
119 0.219371152735326
179 0.141803563561275
270 0.0944625114792827
408 0.0626189410645649
614 0.0412011371514347
925 0.026958737385723
1394 0.0179611940142668
2099 0.0120670218938718
3162 0.00793278120953495
};
\addlegendentry{Natural MSE SCM}
\addplot [semithick, darkorange25512714, mark=asterisk, mark size=3, mark options={solid}]
table {%
10 2.5
15 1.66666666666667
23 1.08695652173913
34 0.735294117647059
52 0.480769230769231
79 0.316455696202532
119 0.210084033613445
179 0.139664804469274
270 0.0925925925925926
408 0.0612745098039216
614 0.0407166123778502
925 0.027027027027027
1394 0.0179340028694405
2099 0.0119104335397809
3162 0.00790638836179633
};
\addlegendentry{Natural CRB}
\end{axis}

\end{tikzpicture}}
    \caption{Intrinsic Cramér-Rao bound}
    \label{fig:NaturCRB}  
    \end{subfigure}
    \caption{Euclidean Cramér-Rao bound and MSE (left) and intrinsic Cramér-Rao bound and distance w.r.t. $n$. The data size is $p=5$.}
\end{figure}
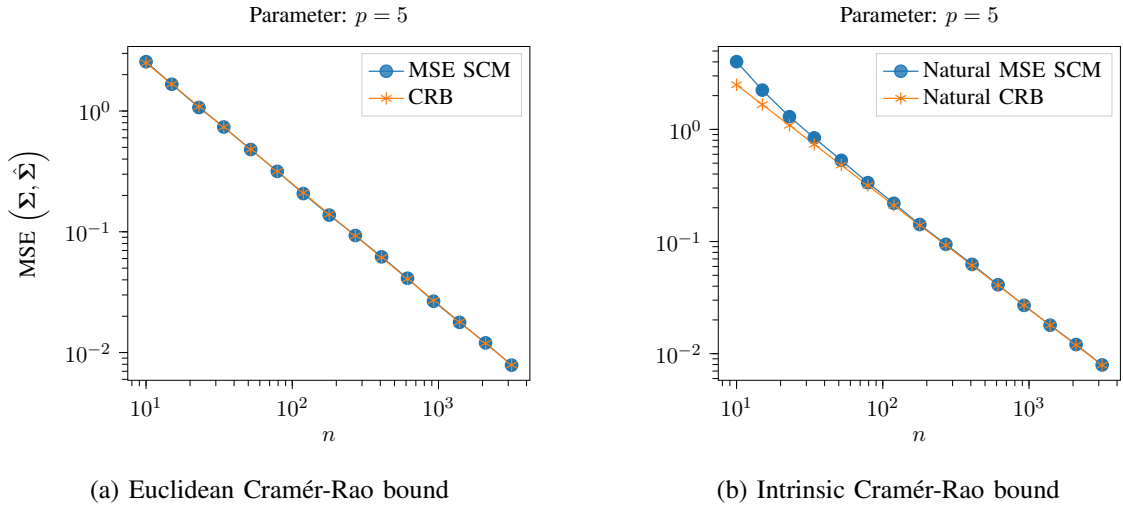

Now, we turn on the study of the MAP and MMSE estimators as well as the corresponding Bayesian Cramér-Rao bounds (BCRB). In particular, we also show the asymptotic Bayesian bounds which are computed without considering the term $\FisherMat_{prior}$ 
(i.e., only ${\rm Tr}(\mathbb{E}_{\mathbf{\Sigma}} \left[ \mathbf{F}_\mathbf{\Sigma} \right]^{-1})$) in \eqref{eq:fim_euclidean} and \eqref{eq:FIM_AI}. We consider two cases: $\nu=40$ and $\nu=100$. The results are shown in Figures \ref{fig:euclidCRB 40} and \ref{fig:NaturCRB 40} for the first value of $\nu$ and in Figures \ref{fig:euclidCRB 100} and \ref{fig:NaturCRB 100} for the second one. For both values, conclusions are similar. We have the property of non efficiency for both the MAP and MMSE estimators in the Euclidean study. Concerning the MAP estimator, this result was expected from \cite{vantrees2007} p.7 since the expression of the classical (non-Bayesian) Fisher information matrix $\FisherMat_{\boldsymbol{\Sigma}}(\boldsymbol{\Sigma})$ depends on $\mathbf{\Sigma}$. In the Riemannian framework, the asymptotic efficiency is achieved, which was also expected for the MAP estimator for the same aforementioned reason since, in this case, the intrinsic non-Bayesian Fisher information matrix is proven to be independent of the studied parameter $\mathbf{\Sigma}$.
This intrinsic analysis allows to conclude that the MAP and MMSE estimators are valid estimators as soon as we have enough samples to estimate them (depending on $\nu$). 
As expected, the MMSE estimator performs slightly better than the MAP in particular for a low sample support.  

\begin{figure}[h]
	\centering
	\begin{subfigure}[b]{0.49\textwidth}
	\centering
	\scalebox{0.78}{% This file was created with tikzplotlib v0.10.1.
\begin{tikzpicture}

\definecolor{crimson2143940}{RGB}{214,39,40}
\definecolor{darkgray176}{RGB}{176,176,176}
\definecolor{darkorange25512714}{RGB}{255,127,14}
\definecolor{forestgreen4416044}{RGB}{44,160,44}
\definecolor{lightgray204}{RGB}{204,204,204}
\definecolor{steelblue31119180}{RGB}{31,119,180}

\begin{axis}[
legend cell align={left},
legend style={fill opacity=0.8, draw opacity=1, text opacity=1, draw=lightgray204},
log basis x={10},
log basis y={10},
tick align=outside,
tick pos=left,
title={Bayesian Euclidean CRB. Parameters: \(\displaystyle p=5\), \(\displaystyle \nu=40\)},
x grid style={darkgray176},
xlabel={\(\displaystyle n\)},
xmin=7.49897501664337, xmax=4216.57625606459,
xmode=log,
xtick style={color=black},
xtick={0.1,1,10,100,1000,10000,100000},
xticklabels={
  \(\displaystyle {10^{-1}}\),
  \(\displaystyle {10^{0}}\),
  \(\displaystyle {10^{1}}\),
  \(\displaystyle {10^{2}}\),
  \(\displaystyle {10^{3}}\),
  \(\displaystyle {10^{4}}\),
  \(\displaystyle {10^{5}}\)
},
y grid style={darkgray176},
ylabel={MSE \(\displaystyle \left( \mathbf{\Sigma}, \hat{\mathbf{\Sigma}} \right)\)},
ymin=0.00444216023115641, ymax=2.53735282112441,
ymode=log,
ytick style={color=black},
ytick={0.0001,0.001,0.01,0.1,1,10,100},
yticklabels={
  \(\displaystyle {10^{-4}}\),
  \(\displaystyle {10^{-3}}\),
  \(\displaystyle {10^{-2}}\),
  \(\displaystyle {10^{-1}}\),
  \(\displaystyle {10^{0}}\),
  \(\displaystyle {10^{1}}\),
  \(\displaystyle {10^{2}}\)
}
]
\addplot [semithick, steelblue31119180, mark=*, mark size=3, mark options={solid}]
table {%
10 0.751379456043145
15 0.67559552265484
23 0.561655458588129
34 0.443447378588238
52 0.343799536410937
79 0.258393170458653
119 0.182050804474421
179 0.136878635731023
270 0.088816080212766
408 0.0593182550138114
614 0.0393435938827832
925 0.0273598020947594
1394 0.0182344159859937
2099 0.0117809951895802
3162 0.00773418801967776
};
\addlegendentry{MAP}
\addplot [semithick, darkorange25512714, mark=square*, mark size=3, mark options={solid}]
table {%
10 0.51465131506076
15 0.474917030434369
23 0.410163454682901
34 0.340629221849511
52 0.281499923399891
79 0.22293236164763
119 0.160114751218902
179 0.122823902537269
270 0.0835974817623273
408 0.0571652538212927
614 0.0386154433451008
925 0.0272925219334481
1394 0.0182297530007147
2099 0.0118150184608622
3162 0.00772505630564744
};
\addlegendentry{MMSE}
\addplot [semithick, forestgreen4416044, mark=asterisk, mark size=3, mark options={solid}]
table {%
10 0.342836403208648
15 0.314460571039329
23 0.27769195084724
34 0.239235611607624
52 0.195043489675539
79 0.152730600782474
119 0.115586070864301
179 0.0846922732955443
270 0.0602639604837784
408 0.0419257543666823
614 0.0288300940288461
925 0.0195915213361518
1394 0.0132085292037243
2099 0.00886629037509615
3162 0.00592792504209307
};
\addlegendentry{BCRB}
\addplot [semithick, crimson2143940]
table {%
10 1.90139512803819
15 1.26759675202546
23 0.82669353392965
34 0.559233861187704
52 0.365652909238114
79 0.240682927599772
119 0.159781103196487
179 0.106223191510514
270 0.0704220417791924
408 0.046602821765642
614 0.0309673473621856
925 0.0205556230058183
1394 0.0136398502728708
2099 0.0090585761221448
3162 0.00601326732459897
};
\addlegendentry{BCRB-Asymptotic}
\end{axis}

\end{tikzpicture}}
    \caption{Euclidean CRB}
    \label{fig:euclidCRB 40}  
    \end{subfigure}
    \begin{subfigure}[b]{0.49\textwidth}
	\centering
	\scalebox{0.78}{% This file was created with tikzplotlib v0.10.1.
\begin{tikzpicture}

\definecolor{crimson2143940}{RGB}{214,39,40}
\definecolor{darkgray176}{RGB}{176,176,176}
\definecolor{darkorange25512714}{RGB}{255,127,14}
\definecolor{forestgreen4416044}{RGB}{44,160,44}
\definecolor{lightgray204}{RGB}{204,204,204}
\definecolor{steelblue31119180}{RGB}{31,119,180}

\begin{axis}[
legend cell align={left},
legend style={fill opacity=0.8, draw opacity=1, text opacity=1, draw=lightgray204},
log basis x={10},
log basis y={10},
tick align=outside,
tick pos=left,
title={Bayesian Intrinsic CRB. Parameters: \(\displaystyle p=5\), \(\displaystyle \nu=40\)},
x grid style={darkgray176},
xlabel={\(\displaystyle n\)},
xmin=7.49897501664337, xmax=4216.57625606459,
xmode=log,
xtick style={color=black},
xtick={0.1,1,10,100,1000,10000,100000},
xticklabels={
  \(\displaystyle {10^{-1}}\),
  \(\displaystyle {10^{0}}\),
  \(\displaystyle {10^{1}}\),
  \(\displaystyle {10^{2}}\),
  \(\displaystyle {10^{3}}\),
  \(\displaystyle {10^{4}}\),
  \(\displaystyle {10^{5}}\)
},
y grid style={darkgray176},
ylabel={MSE \(\displaystyle \left( \mathbf{\Sigma}, \hat{\mathbf{\Sigma}} \right)\)},
ymin=0.00582577082867704, ymax=3.33657795391498,
ymode=log,
ytick style={color=black},
ytick={0.0001,0.001,0.01,0.1,1,10,100},
yticklabels={
  \(\displaystyle {10^{-4}}\),
  \(\displaystyle {10^{-3}}\),
  \(\displaystyle {10^{-2}}\),
  \(\displaystyle {10^{-1}}\),
  \(\displaystyle {10^{0}}\),
  \(\displaystyle {10^{1}}\),
  \(\displaystyle {10^{2}}\)
}
]
\addplot [semithick, steelblue31119180, mark=*, mark size=3, mark options={solid}]
table {%
10 0.623048129618166
15 0.567493551917767
23 0.468031714777544
34 0.386440258850185
52 0.30598103130302
79 0.232027216772537
119 0.164179563592146
179 0.121547987761809
270 0.0830487543260053
408 0.0578731723218627
614 0.0385494287632879
925 0.0270143764903961
1394 0.0175756070769067
2099 0.0118690002844449
3162 0.00779996163067527
};
\addlegendentry{MAP}
\addplot [semithick, darkorange25512714, mark=square*, mark size=3, mark options={solid}]
table {%
10 0.530916408601205
15 0.483931636335165
23 0.411202356907676
34 0.344971505506028
52 0.2799157504194
79 0.214708966967417
119 0.155619556050866
179 0.116648102479801
270 0.0810886859242845
408 0.0569194655334584
614 0.0381053116956061
925 0.0268296131743916
1394 0.017587845741783
2099 0.0118543059301078
3162 0.00777525540460993
};
\addlegendentry{MMSE}
\addplot [semithick, forestgreen4416044, mark=asterisk, mark size=3, mark options={solid}]
table {%
10 0.504966154875219
15 0.457585013055797
23 0.397942694407041
34 0.337570656330562
52 0.270575521294457
79 0.208655066667263
119 0.15595682320141
179 0.113207996892705
270 0.0800318886724213
408 0.0554464682790147
614 0.0380330905411882
925 0.0258083884229743
1394 0.0173858651768473
2099 0.0116650745874889
3162 0.00779716823188569
};
\addlegendentry{BICRB}
\addplot [semithick, crimson2143940]
table {%
10 2.5
15 1.66666666666667
23 1.08695652173913
34 0.735294117647059
52 0.480769230769231
79 0.316455696202532
119 0.210084033613445
179 0.139664804469274
270 0.0925925925925926
408 0.0612745098039216
614 0.0407166123778502
925 0.027027027027027
1394 0.0179340028694405
2099 0.0119104335397808
3162 0.00790638836179633
};
\addlegendentry{BICRB-Asymptotic}
\end{axis}

\end{tikzpicture}}
    \caption{Intrinsic CRB}
    \label{fig:NaturCRB 40}  
    \end{subfigure}
    \caption{Euclidean Bayesian Cramér-Rao bound and MSE (left) and intrinsic Bayesian Cramér-Rao bound and expectation of the natural distance w.r.t. $n$. The data size is $p=5$ and the number of degrees of freedom of the $\mathcal{IW}$ prior is $\nu=40$.}
\end{figure}
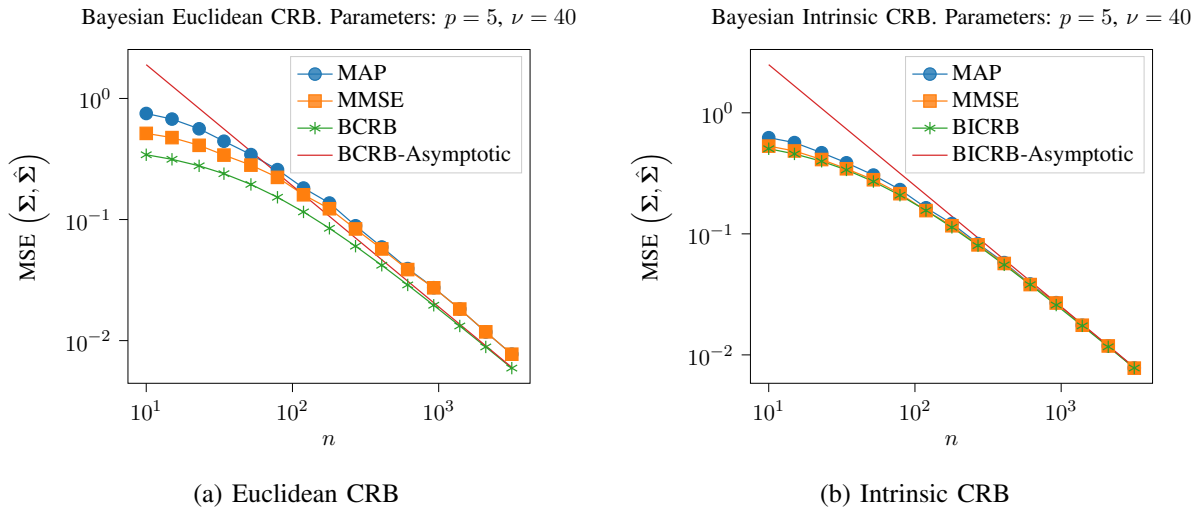

\begin{figure}[h]
	\centering
	\begin{subfigure}[b]{0.49\textwidth}
	\centering
	\scalebox{0.78}{% This file was created with tikzplotlib v0.10.1.
\begin{tikzpicture}

\definecolor{crimson2143940}{RGB}{214,39,40}
\definecolor{darkgray176}{RGB}{176,176,176}
\definecolor{darkorange25512714}{RGB}{255,127,14}
\definecolor{forestgreen4416044}{RGB}{44,160,44}
\definecolor{lightgray204}{RGB}{204,204,204}
\definecolor{steelblue31119180}{RGB}{31,119,180}

\begin{axis}[
legend cell align={left},
legend style={fill opacity=0.8, draw opacity=1, text opacity=1, draw=lightgray204},
log basis x={10},
log basis y={10},
tick align=outside,
tick pos=left,
title={Bayesian Euclidean CRB. Parameters: \(\displaystyle p=5\), \(\displaystyle \nu=100\)},
x grid style={darkgray176},
xlabel={\(\displaystyle n\)},
xmin=7.49897501664337, xmax=4216.57625606459,
xmode=log,
xtick style={color=black},
xtick={0.1,1,10,100,1000,10000,100000},
xticklabels={
  \(\displaystyle {10^{-1}}\),
  \(\displaystyle {10^{0}}\),
  \(\displaystyle {10^{1}}\),
  \(\displaystyle {10^{2}}\),
  \(\displaystyle {10^{3}}\),
  \(\displaystyle {10^{4}}\),
  \(\displaystyle {10^{5}}\)
},
y grid style={darkgray176},
ylabel={MSE \(\displaystyle \left( \mathbf{\Sigma}, \hat{\mathbf{\Sigma}} \right)\)},
ymin=0.00515488393696443, ymax=3.00514101343337,
ymode=log,
ytick style={color=black},
ytick={0.0001,0.001,0.01,0.1,1,10,100},
yticklabels={
  \(\displaystyle {10^{-4}}\),
  \(\displaystyle {10^{-3}}\),
  \(\displaystyle {10^{-2}}\),
  \(\displaystyle {10^{-1}}\),
  \(\displaystyle {10^{0}}\),
  \(\displaystyle {10^{1}}\),
  \(\displaystyle {10^{2}}\)
}
]
\addplot [semithick, steelblue31119180, mark=*, mark size=3, mark options={solid}]
table {%
10 0.273338402655649
15 0.261713070773781
23 0.240376044847161
34 0.22211738797246
52 0.198122349693748
79 0.167649749792183
119 0.139375075007213
179 0.0996850011314294
270 0.0726467177962272
408 0.0522834660786607
614 0.0362519412034416
925 0.0244698218245175
1394 0.0170599077322346
2099 0.0117121274292034
3162 0.00782692050008656
};
\addlegendentry{MAP}
\addplot [semithick, darkorange25512714, mark=square*, mark size=3, mark options={solid}]
table {%
10 0.218902803904826
15 0.21065590081379
23 0.195228785922631
34 0.182920146710207
52 0.165984850520781
79 0.142484317707535
119 0.121999796820617
179 0.0899032499597853
270 0.0677613072777615
408 0.0495189925740192
614 0.0350568567329667
925 0.0241919153256901
1394 0.0170242658401078
2099 0.0116830579333461
3162 0.00781121743986558
};
\addlegendentry{MMSE}
\addplot [semithick, forestgreen4416044, mark=asterisk, mark size=3, mark options={solid}]
table {%
10 0.194754326165918
15 0.186672386655143
23 0.175049931499484
34 0.161246254512756
52 0.142818255306595
79 0.121918986049432
119 0.100197829044096
179 0.0790683222394613
270 0.0599082997611641
408 0.0438095611470119
614 0.0312672474545442
925 0.021831402525361
1394 0.0150034197951689
2099 0.0102054425954898
3162 0.0068854142134261
};
\addlegendentry{BCRB}
\addplot [semithick, crimson2143940]
table {%
10 2.2498505766369
15 1.4999003844246
23 0.97819590288561
34 0.661720757834383
52 0.432663572430174
79 0.28479121223252
119 0.189063073666967
179 0.12568997634843
270 0.0833277991347002
408 0.0551433964861987
614 0.036642517534803
925 0.0243227089366152
1394 0.0161395306788874
2099 0.0107186783069886
3162 0.00711527696596111
};
\addlegendentry{BCRB-Asymptotic}
\end{axis}

\end{tikzpicture}}
    \caption{Euclidean CRB}
    \label{fig:euclidCRB 100}  
    \end{subfigure}
    \begin{subfigure}[b]{0.49\textwidth}
	\centering
	\scalebox{0.78}{% This file was created with tikzplotlib v0.10.1.
\begin{tikzpicture}

\definecolor{crimson2143940}{RGB}{214,39,40}
\definecolor{darkgray176}{RGB}{176,176,176}
\definecolor{darkorange25512714}{RGB}{255,127,14}
\definecolor{forestgreen4416044}{RGB}{44,160,44}
\definecolor{lightgray204}{RGB}{204,204,204}
\definecolor{steelblue31119180}{RGB}{31,119,180}

\begin{axis}[
legend cell align={left},
legend style={fill opacity=0.8, draw opacity=1, text opacity=1, draw=lightgray204},
log basis x={10},
log basis y={10},
tick align=outside,
tick pos=left,
title={Bayesian Intrinsic CRB. Parameters: \(\displaystyle p=5\), \(\displaystyle \nu=100\)},
x grid style={darkgray176},
xlabel={\(\displaystyle n\)},
xmin=7.49897501664337, xmax=4216.57625606459,
xmode=log,
xtick style={color=black},
xtick={0.1,1,10,100,1000,10000,100000},
xticklabels={
  \(\displaystyle {10^{-1}}\),
  \(\displaystyle {10^{0}}\),
  \(\displaystyle {10^{1}}\),
  \(\displaystyle {10^{2}}\),
  \(\displaystyle {10^{3}}\),
  \(\displaystyle {10^{4}}\),
  \(\displaystyle {10^{5}}\)
},
y grid style={darkgray176},
ylabel={MSE \(\displaystyle \left( \mathbf{\Sigma}, \hat{\mathbf{\Sigma}} \right)\)},
ymin=0.00573035414636566, ymax=3.33920280518225,
ymode=log,
ytick style={color=black},
ytick={0.0001,0.001,0.01,0.1,1,10,100},
yticklabels={
  \(\displaystyle {10^{-4}}\),
  \(\displaystyle {10^{-3}}\),
  \(\displaystyle {10^{-2}}\),
  \(\displaystyle {10^{-1}}\),
  \(\displaystyle {10^{0}}\),
  \(\displaystyle {10^{1}}\),
  \(\displaystyle {10^{2}}\)
}
]
\addplot [semithick, steelblue31119180, mark=*, mark size=3, mark options={solid}]
table {%
10 0.243396863264198
15 0.234770448989664
23 0.216813759720282
34 0.20190035740599
52 0.179215253577428
79 0.152166032597536
119 0.124458307060591
179 0.0925543723476403
270 0.0690377884268886
408 0.0493475582183671
614 0.035629667176261
925 0.0243446066491116
1394 0.016714081402282
2099 0.0115034022655472
3162 0.00774019691493689
};
\addlegendentry{MAP}
\addplot [semithick, darkorange25512714, mark=square*, mark size=3, mark options={solid}]
table {%
10 0.224539199655797
15 0.216219629585279
23 0.201147379599437
34 0.187915673704148
52 0.168581340633972
79 0.144127550855517
119 0.119414178254657
179 0.089811310406863
270 0.0676622823401337
408 0.0485372490975668
614 0.0352727220995529
925 0.0241951495127254
1394 0.016751942265298
2099 0.0114884338342866
3162 0.00771625184936977
};
\addlegendentry{MMSE}
\addplot [semithick, forestgreen4416044, mark=asterisk, mark size=3, mark options={solid}]
table {%
10 0.225866225689131
15 0.215967242004456
23 0.201824545983004
34 0.18516708027601
52 0.163158701216637
79 0.138504583004349
119 0.113207996892705
179 0.0888995754097199
270 0.0670906463210362
408 0.0489197831576024
614 0.0348475467530787
925 0.0243019451680333
1394 0.0166892150263719
2099 0.0113473318602785
3162 0.00765392585609278
};
\addlegendentry{BICRB}
\addplot [semithick, crimson2143940]
table {%
10 2.5
15 1.66666666666667
23 1.08695652173913
34 0.735294117647059
52 0.480769230769231
79 0.316455696202532
119 0.210084033613445
179 0.139664804469274
270 0.0925925925925926
408 0.0612745098039216
614 0.0407166123778502
925 0.027027027027027
1394 0.0179340028694405
2099 0.0119104335397808
3162 0.00790638836179633
};
\addlegendentry{BICRB-Asymptotic}
\end{axis}

\end{tikzpicture}}
    \caption{Intrinsic CRB}
    \label{fig:NaturCRB 100}  
    \end{subfigure}
    \caption{Euclidean Bayesian Cramér-Rao bound and MSE (left) and intrinsic Bayesian Cramér-Rao bound and expectation of the natural distance w.r.t. $n$. The data size is $p=5$ and the number of degrees of freedom of the $\mathcal{IW}$ prior is $\nu=100$.}
\end{figure}
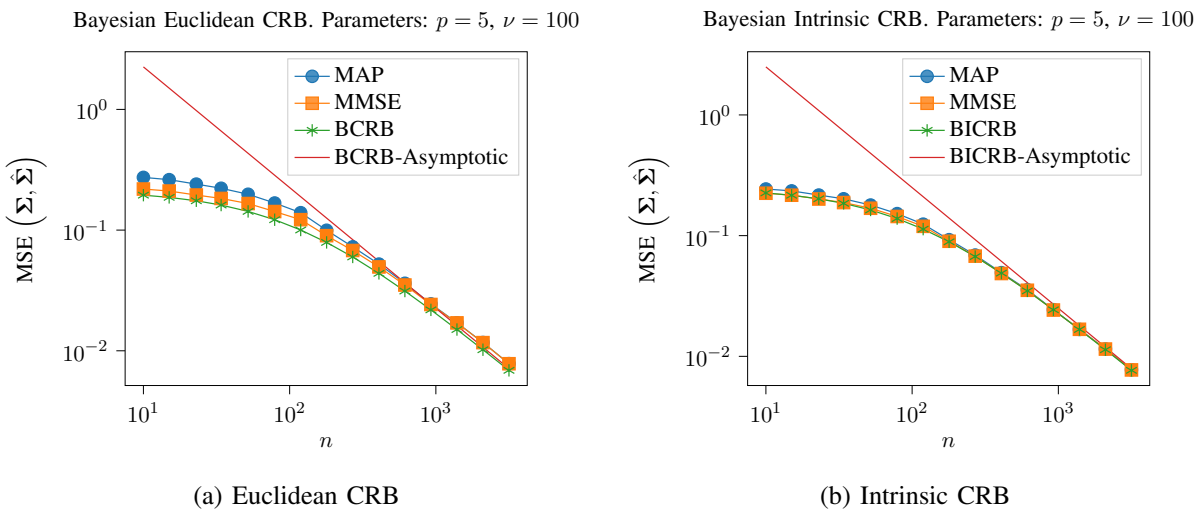

\section{Conclusion}

In the context of Bayesian estimation, when the parameter to estimate lies in a manifold, we have proposed a new intrinsic Van Trees inequality between a covariance of the estimation error defined with geometric tools and an intrinsic Bayesian Fisher information. This derivation is made by using some assumptions on the manifold of interest and the prior distribution. We illustrated this result by considering the problem of covariance estimation when the data follow a Gaussian distribution and the prior distribution is an inverse Wishart. Numerical simulation leaded to interesting conclusions on the MAP and the MMSE estimators which seem to be asymptotic efficiency in the natural inequality which is not the case when the study is made with the Euclidean formalism.
  
\bibliographystyle{IEEEbib}
\bibliography{biblio}

\appendices

\section{Expectations \ref{eq:all_Ts} for Theorem \ref{thm:eucl_bcrb_iw} (Euclidean metric)}
\label{appen:proof_euclidian}

This proof is based on the computation of expectations of traces of functions of Wishart matrices. Main results have been derived in \cite{maiwald00}, however they do not cover the higher orders that we have to deal with here. The required results are obtained thanks to the following lemma:

\begin{lemma}
Let $\boldsymbol{S} \sim \mathcal{W}(K,\frac{1}{K}\boldsymbol{\Sigma})$ which follows a Wishart distribution of scale matrix $\boldsymbol{\Sigma}$ and degrees of freedom $K$, then we have the following expectations
\begin{equation}
\begin{array}{l}
\mathbb{E}[\tr(\boldsymbol{A}\boldsymbol{S}\boldsymbol{B}\boldsymbol{S})\tr(\boldsymbol{C}\boldsymbol{S})] = \tr(\boldsymbol{A}\boldsymbol{\Sigma}\boldsymbol{B}\boldsymbol{\Sigma})\tr(\boldsymbol{C}\boldsymbol{\Sigma}) \\ +\frac{1}{K}(\tr(\boldsymbol{A}\boldsymbol{\Sigma}\boldsymbol{B}\boldsymbol{\Sigma}\boldsymbol{C}\boldsymbol{\Sigma})+\tr(\boldsymbol{A}\boldsymbol{\Sigma}\boldsymbol{C}\boldsymbol{\Sigma}\boldsymbol{B}\boldsymbol{\Sigma})+\tr(\boldsymbol{A}\boldsymbol{\Sigma})\tr(\boldsymbol{B}\boldsymbol{\Sigma})\tr(\boldsymbol{C}\boldsymbol{\Sigma}) ) \\
+\frac{1}{K^2} (\tr(\boldsymbol{A}\boldsymbol{\Sigma}\boldsymbol{C}\boldsymbol{\Sigma})\tr(\boldsymbol{B}\boldsymbol{\Sigma})+\tr(\boldsymbol{B}\boldsymbol{\Sigma}\boldsymbol{C}\boldsymbol{\Sigma})\tr(\boldsymbol{A}\boldsymbol{\Sigma})),
\end{array}
\end{equation}
and 
\begin{equation}
\begin{array}{l}
\mathbb{E}[\tr(\boldsymbol{A}\boldsymbol{S}\boldsymbol{B}\boldsymbol{S})\tr(\boldsymbol{C}\boldsymbol{S}\mathbf{D}\boldsymbol{S})] = \tr(\boldsymbol{A}\boldsymbol{\Sigma}\boldsymbol{B}\boldsymbol{\Sigma})\tr(\boldsymbol{C}\boldsymbol{\Sigma}\mathbf{D}\boldsymbol{\Sigma}) \\
+ \frac{1}{K}(\tr(\boldsymbol{A}\boldsymbol{\Sigma}\boldsymbol{B}\boldsymbol{\Sigma}\boldsymbol{C}\boldsymbol{\Sigma}\mathbf{D}\boldsymbol{\Sigma})+\tr(\boldsymbol{A}\boldsymbol{\Sigma}\boldsymbol{C}\boldsymbol{\Sigma}\mathbf{D}\boldsymbol{\Sigma}\boldsymbol{B}\boldsymbol{\Sigma})+\tr(\boldsymbol{A}\boldsymbol{\Sigma}\boldsymbol{B}\boldsymbol{\Sigma}\mathbf{D}\boldsymbol{\Sigma}\boldsymbol{C}\boldsymbol{\Sigma})\\
+\tr(\boldsymbol{A}\boldsymbol{\Sigma}\mathbf{D}\boldsymbol{\Sigma}\boldsymbol{C}\boldsymbol{\Sigma}\boldsymbol{B}\boldsymbol{\Sigma})+\tr(\boldsymbol{A}\boldsymbol{\Sigma}\boldsymbol{B}\boldsymbol{\Sigma})\tr(\boldsymbol{C}\boldsymbol{\Sigma})\tr(\mathbf{D}\boldsymbol{\Sigma})+\tr(\boldsymbol{C}\boldsymbol{\Sigma}\mathbf{D}\boldsymbol{\Sigma})\tr(\boldsymbol{A}\boldsymbol{\Sigma})\tr(\boldsymbol{B}\boldsymbol{\Sigma})) \\
\frac{1}{K^2}(\tr(\boldsymbol{A}\boldsymbol{\Sigma}\mathbf{D}\boldsymbol{\Sigma}\boldsymbol{B}\boldsymbol{\Sigma})\tr(\boldsymbol{C}\boldsymbol{\Sigma})+\tr(\boldsymbol{A}\boldsymbol{\Sigma}\boldsymbol{C}\boldsymbol{\Sigma}\boldsymbol{B}\boldsymbol{\Sigma})\tr(\mathbf{D}\boldsymbol{\Sigma})+\tr(\boldsymbol{A}\boldsymbol{\Sigma}\mathbf{D}\boldsymbol{\Sigma})\tr(\boldsymbol{B}\boldsymbol{\Sigma}\boldsymbol{C}\boldsymbol{\Sigma}) \\
+ \tr(\boldsymbol{A}\boldsymbol{\Sigma}\boldsymbol{B}\boldsymbol{\Sigma}\mathbf{D}\boldsymbol{\Sigma})\tr(\boldsymbol{C}\boldsymbol{\Sigma}) + \tr(\boldsymbol{A}\boldsymbol{\Sigma}\boldsymbol{B}\boldsymbol{\Sigma}\boldsymbol{C}\boldsymbol{\Sigma})\tr(\mathbf{D}\boldsymbol{\Sigma})+\tr(\boldsymbol{A}\boldsymbol{\Sigma}\boldsymbol{C}\boldsymbol{\Sigma})\tr(\boldsymbol{B}\boldsymbol{\Sigma}\mathbf{D}\boldsymbol{\Sigma}) \\
+\tr(\boldsymbol{A}\boldsymbol{\Sigma})\tr(\boldsymbol{B}\boldsymbol{\Sigma})\tr(\boldsymbol{C}\boldsymbol{\Sigma})\tr(\mathbf{D}\boldsymbol{\Sigma})+\tr(\boldsymbol{B}\boldsymbol{\Sigma}\boldsymbol{C}\boldsymbol{\Sigma}\mathbf{D}\boldsymbol{\Sigma})\tr(\boldsymbol{A}\boldsymbol{\Sigma})+\tr(\boldsymbol{A}\boldsymbol{\Sigma}\boldsymbol{C}\boldsymbol{\Sigma}\mathbf{D}\boldsymbol{\Sigma})\tr(\boldsymbol{B}\boldsymbol{\Sigma}) \\
+ \tr(\boldsymbol{A}\boldsymbol{\Sigma}\mathbf{D}\boldsymbol{\Sigma}\boldsymbol{C}\boldsymbol{\Sigma})\tr(\boldsymbol{B}\boldsymbol{\Sigma}) + \tr(\boldsymbol{B}\boldsymbol{\Sigma}\mathbf{D}\boldsymbol{\Sigma}\boldsymbol{C}\boldsymbol{\Sigma})\tr(\boldsymbol{A}\boldsymbol{\Sigma})) \\
+\frac{1}{K^3}(\tr(\boldsymbol{A}\boldsymbol{\Sigma}\mathbf{D}\boldsymbol{\Sigma}\boldsymbol{B}\boldsymbol{\Sigma}\boldsymbol{C}\boldsymbol{\Sigma})+\tr(\boldsymbol{A}\boldsymbol{\Sigma}\boldsymbol{C}\boldsymbol{\Sigma})\tr(\boldsymbol{B}\boldsymbol{\Sigma})\tr(\mathbf{D}\boldsymbol{\Sigma})+\tr(\boldsymbol{A}\boldsymbol{\Sigma}\mathbf{D}\boldsymbol{\Sigma})\tr(\boldsymbol{B}\boldsymbol{\Sigma})\tr(\boldsymbol{C}\boldsymbol{\Sigma}) \\
+\tr(\boldsymbol{B}\boldsymbol{\Sigma}\boldsymbol{C}\boldsymbol{\Sigma})\tr(\boldsymbol{A}\boldsymbol{\Sigma})\tr(\mathbf{D}\boldsymbol{\Sigma}) + \tr(\boldsymbol{B}\boldsymbol{\Sigma}\mathbf{D}\boldsymbol{\Sigma})\tr(\boldsymbol{A}\boldsymbol{\Sigma})\tr(\boldsymbol{C}\boldsymbol{\Sigma}) + \tr(\boldsymbol{A}\boldsymbol{\Sigma}\boldsymbol{C}\boldsymbol{\Sigma}\boldsymbol{B}\boldsymbol{\Sigma}\mathbf{D}\boldsymbol{\Sigma}))
\end{array} .
\end{equation}
\label{lemma}
\end{lemma}
\begin{proof}
    The derivations are very long and are therefore omitted in this paper. Nevertheless, they can be retrieved by using the method of \cite{maiwald00} (in particular the relation $\tr(\boldsymbol{A}_{(1)} \hdots \boldsymbol{A}_{(n)})= \sum_{i_1,\hdots,i_n} a^{(1)}_{i_1i_2} \hdots a^{(n)}_{i_ni_1}$) but considering a higher order than in this reference.
\end{proof}

Let go back to the proof of Theorem \ref{thm:eucl_bcrb_iw}. We recall that we have to compute the following expectations:

\begin{itemize}
\item $T_1(i,j)=\mathbb{E}\left[\tr(\boldsymbol{\Sigma}^{-1}\boldsymbol{\Omega}_i^E)\tr(\boldsymbol{\Sigma}^{-1}\boldsymbol{\Omega}^E_j)\right]$,
\item $T_2(i,j)=\mathbb{E}\left[\tr(\boldsymbol{\Sigma}^{-1}\boldsymbol{\Omega}^E_i\boldsymbol{\Sigma}^{-1}\boldsymbol{\Sigma}_0)\tr(\boldsymbol{\Sigma}^{-1}\boldsymbol{\Omega}^E_j\boldsymbol{\Sigma}^{-1}\boldsymbol{\Sigma}_0)\right]$,
\item $T_3(i,j)=\mathbb{E}\left[\tr(\boldsymbol{\Sigma}^{-1}\boldsymbol{\Omega}^E_j)\tr(\boldsymbol{\Sigma}^{-1}\boldsymbol{\Omega}^E_i\boldsymbol{\Sigma}^{-1}\boldsymbol{\Sigma}_0)\right]$,
\item $T_4(i,j)=\mathbb{E}\left[\tr\left[\boldsymbol{\Sigma}^{-1} \boldsymbol{\Omega}^E_i\boldsymbol{\Sigma}^{-1}\boldsymbol{\Omega}^E_j\right]\right]$.
\end{itemize}
Since $\boldsymbol{\Sigma}\sim \mathcal{IW}(\nu,(\nu-p)\boldsymbol{\Sigma}_0)$, it is easy to notice that $\boldsymbol{\Sigma}^{-1}\sim \mathcal{W}(\nu,\frac{\boldsymbol{\Sigma}_0^{-1}}{\nu-p})$. Therefore, we can apply the results of Lemma \ref{lemma}, but taking into account the scale of the problem considered in our case. 

Let us introduce the changes of variables $\tilde{\boldsymbol{\Sigma}}^{-1}=\frac{1}{\nu}\boldsymbol{\Sigma}^{-1}$ and $\tilde{\boldsymbol{\Sigma}}^{-1}_0=\frac{1}{\nu-p}\boldsymbol{\Sigma}^{-1}_0$. In this case, we have $\tilde{\boldsymbol{\Sigma}}^{-1} \sim \mathcal{W}(\nu,\frac{\tilde{\boldsymbol{\Sigma}}^{-1}}{\nu})$. We are then in the conditions of the Lemma \ref{lemma} which can be applied to each term involving $\boldsymbol{\Sigma}^{-1}$ in \eqref{eq:all_Ts}, then properly re-scaled by $\nu/(\nu-p)$ to have the final expectations. 
We finally obtain for the 4 terms:
\begin{equation}
    \begin{array}{lll}
         T_1(i,j) & = &  \frac{1}{(\nu-p)^2}(\nu^2 A + \nu B)\\
         T_2(i,j) & = & \frac{1}{(\nu-p)^4}(\nu^4 A + \nu^3(4B+2pA) + \nu^2 (5A+5pB+p^2A) +\nu(3pA+2B+p^2B))\\
         T_3(i,j) &=& \frac{1}{(\nu-p)^3}(\nu^3 A + \nu^2(2B+pA) \nu (pB+A)) \\
         T_4(i,j) &=& \frac{1}{(\nu-p)^2}(\nu^2 B + \nu A) \\
    \end{array},
\end{equation}
where $A=\tr(\boldsymbol{\Sigma}_0^{-1}\boldsymbol{\Omega}_i^E)\tr(\boldsymbol{\Sigma}_0^{-1}\boldsymbol{\Omega}^E_j)$ and $B=\tr(\boldsymbol{\Sigma}_0^{-1}\boldsymbol{\Omega}_i^E\boldsymbol{\Sigma}_0^{-1}\boldsymbol{\Omega}^E_j)$. Finally, after basic manipulations, we have the following result:
\begin{equation}
    \left(\FisherMat^{\mathcal{E}}_{\textup{prior}}\right)_{i,j} = \alpha  \tr(\boldsymbol{\Sigma}_0^{-1}\boldsymbol{\Omega}_i^E\boldsymbol{\Sigma}_0^{-1}\boldsymbol{\Omega}^E_j) + \beta \tr(\boldsymbol{\Sigma}_0^{-1}\boldsymbol{\Omega}_i^E)\tr(\boldsymbol{\Sigma}_0^{-1}\boldsymbol{\Omega}^E_j),
\end{equation}
where $\alpha = \frac{\nu^3+p\nu^2+2\nu}{(\nu-p)^2}$ and $\beta = \frac{3\nu^2-p\nu}{(\nu-p)^2} $. 
By using again the term $T_4(i,j)$, the final result is factorized as
\begin{equation}
    \left(\FisherMat^\mathcal{E}_{\mathcal{IW}}\right)_{i,j} = \frac{n\nu^2}{(\nu-p)^2} \tr(\boldsymbol{\Sigma}_0^{-1}\boldsymbol{\Omega}_i^E\boldsymbol{\Sigma}_0^{-1}\boldsymbol{\Omega}^E_j)+ \frac{n\nu}{(\nu-p)^2} \tr(\boldsymbol{\Sigma}_0^{-1}\boldsymbol{\Omega}_i^E)\tr(\boldsymbol{\Sigma}_0^{-1}\boldsymbol{\Omega}^E_j) + \left(\FisherMat^{\mathcal{E}}_{\textup{prior}}\right)_{i,j}.
\end{equation}

\section{Expectations \ref{eq:all_Ts} for Theorem \ref{thm:riem_bcrb_iw} (Affine invariant metric)}
\label{appen:proof_affine}

This proof is decomposed in three steps. Firstly, we have to choose a basis which has been to be orthonormal w.r.t. the affine invariant metric \eqref{eq:AI metric}. Secondly, from this chosen basis, we reduce the terms $T_1$, $T_2$ and $T_3$ of \eqref{eq:all_Ts} (we will show that the computation of $T_4$ is not needed for this bound). Finally, the expectations for these reduced terms are computed.

\subsection{Choice of the basis}

In this section, we have to make a choice for the basis $\{\boldsymbol{\Omega}_i\}$. This basis has been to be orthonormal w.r.t. the natural inner product $n\tr\left[\boldsymbol{\Sigma}^{-1} \boldsymbol{\Omega}_i\boldsymbol{\Sigma}^{-1}\boldsymbol{\Omega}_j\right]$. As for the Euclidean proof, since $\boldsymbol{\Sigma}\sim\mathcal{IW}(\nu,(\nu-p)\boldsymbol{\Sigma}_{0})$, we have: 
\begin{equation}
\boldsymbol{\Sigma}^{-1}\sim\mathcal{W}(\nu,\frac{1}{\nu-p}\boldsymbol{\Sigma}_{0}^{-1}).
\end{equation}
Let us consider the Cholesky decomposition of $\boldsymbol{\Sigma}_{0}^{-1}$:
\begin{equation}
\boldsymbol{\Sigma}_{0}^{-1}=\boldsymbol{L}\boldsymbol{L}^{H},
\end{equation}
where $\boldsymbol{L}$ is a lower triangular matrix. We also use the Bartlett decomposition of $\boldsymbol{\Sigma}^{-1}$: 
\begin{equation}
\boldsymbol{\Sigma}^{-1}=\frac{1}{\nu-p}\boldsymbol{L}\boldsymbol{A}\boldsymbol{A}^{H}\boldsymbol{L}^{H},\label{eq:bartlett}
\end{equation}
where $\boldsymbol{A}$ is also a lower triangular matrix and where all
the elements are independent random variables: 
\begin{equation}
\begin{array}{lll}
a_{i,j} & \sim & \mathcal{CN}(0,1),\,i>j\\
a_{i,i}^{2} & \sim & \frac{1}{2}\chi_{2(\nu-i+1)}^{2}
\end{array}.
\end{equation}
Let us define a square root of $\boldsymbol{\Sigma}$: 
\begin{equation}
\boldsymbol{\Sigma}=\boldsymbol{H}\boldsymbol{H}^{H}.
\end{equation}
We know there exists an unitary matrix $\mathbf{U}\in\mathcal{U}_{p}$ such that\footnote{$\mathcal{U}_{p}$ is the set of unitary matrices of size $p \times p$}: 
\begin{equation}
\boldsymbol{H}=\boldsymbol{\Sigma}^{1/2}\mathbf{U}.
\end{equation}
Based on this definition of the square root, we choose to build the
natural basis from the Euclidean basis and $\boldsymbol{H}$: 
\begin{equation}
\boldsymbol{\Omega}_{i}=\boldsymbol{H}\boldsymbol{\Omega}_{i}^{E}\boldsymbol{H}^{H}.
\label{eq:defbasisnatural}
\end{equation}
Let us choose now a matrix $\boldsymbol{H}$. From the Bartlett decomposition
\eqref{eq:bartlett}: 
\begin{equation}
\boldsymbol{\Sigma}=(\nu-p)\boldsymbol{L}^{-H}\boldsymbol{A}^{-H}\boldsymbol{A}^{-1}\boldsymbol{L}^{-1}.
\end{equation}
Therefore, we choose 
\begin{equation}
\boldsymbol{H}=\sqrt{\nu-p}\boldsymbol{L}^{-H}\boldsymbol{A}^{-H}.
\label{eq:defH}
\end{equation}
It is easy to see that the basis defined in \eqref{eq:defbasisnatural} with $\boldsymbol{H}$ given as in \eqref{eq:defH} is orthonormal w.r.t. the natural inner product.

\subsection{Computation of the terms $T_1$, $T_2$ and $T_3$ of \eqref{eq:all_Ts}}

We recall that we have to compute the matrix 
\begin{equation}
\FisherMat^{AI}_{\mathcal{IW}} = \mathbb{E}_{\boldsymbol{\Sigma}} \left[\FisherMat_{\boldsymbol{\Sigma}}(\boldsymbol{\Sigma}) \right] +  \FisherMat^{AI}_{\textup{prior}},
\end{equation}
where $\FisherMat_{\boldsymbol{\Sigma}}(\boldsymbol{\Sigma})=n \mathbf{I}_p$ \cite{Breloy2019} and 
\begin{equation}
\left(\FisherMat^{AI}_{\textup{prior}}\right)_{i,j} =  \left[ (\nu+p)^2 T_1(i,j) +(\nu-p)^2 T_2(i,j) - (\nu^2-p^2)T_3(i,j) - (\nu^2-p^2)T_3(j,i)  \right].
\end{equation}
We notice, compared to the Euclidean bound, that we only have to compute $T_1$, $T_2$ and $T_3$ since $\FisherMat_{\boldsymbol{\Sigma}}(\boldsymbol{\Sigma})$ does not involve $\boldsymbol{\Sigma}$.

Let us start to compute the terms $\tr(\boldsymbol{\Sigma}^{-1}\boldsymbol{\Omega}_{i})$ and $\tr(\boldsymbol{\Sigma}^{-1}\boldsymbol{\Omega}_{i}\boldsymbol{\Sigma}^{-1}\boldsymbol{\Sigma}_{0})$ by using the basis chosen in the previous section. For the first term, we have:
\begin{equation}
\begin{array}{lll}
\tr(\boldsymbol{\Sigma}^{-1}\boldsymbol{\Omega}_{i}) & = & \tr(\boldsymbol{\Sigma}^{-1}\boldsymbol{H}\boldsymbol{\Omega}_{i}^{E}\boldsymbol{H}^{H})\\
 & = & \tr(\frac{1}{\nu-p}\boldsymbol{L}\boldsymbol{A}\boldsymbol{A}^{H}\boldsymbol{L}^{H}\boldsymbol{L}^{-H}\boldsymbol{A}^{-H}\sqrt{\nu-p}\boldsymbol{\Omega}_{i}^{E}\sqrt{\nu-p}\boldsymbol{A}^{-1}\boldsymbol{L}^{-1})\\
 & = & \tr(\boldsymbol{\Omega}_{i}^{E}).
\end{array}
\end{equation}
The second term is 
\begin{equation}
\begin{array}{lll}
\tr(\boldsymbol{\Sigma}^{-1}\boldsymbol{\Omega}_{i}\boldsymbol{\Sigma}^{-1}\boldsymbol{\Sigma}_{0}) & = & \tr((\nu-p-1)\boldsymbol{\Sigma}^{-1}\boldsymbol{L}^{-H}\boldsymbol{A}^{-H}\boldsymbol{\Omega}_{i}^{E}\boldsymbol{A}^{-1}\boldsymbol{L}^{-1}\boldsymbol{\Sigma}^{-1}\boldsymbol{\Sigma}_{0})\\
 & = & \tr(\frac{\nu-p}{\nu-p}\boldsymbol{\Sigma}^{-1}\boldsymbol{L}^{-H}\boldsymbol{A}^{-H}\boldsymbol{\Omega}_{i}^{E}\boldsymbol{A}^{-1}\boldsymbol{L}^{-1}\boldsymbol{\Sigma}^{-1}\boldsymbol{L}^{-H}\boldsymbol{L}^{-1})\\
 & = & \tr(\frac{1}{\nu-p}\boldsymbol{L}\boldsymbol{A}\boldsymbol{A}^{H}\boldsymbol{L}^{H}\boldsymbol{L}^{-H}\boldsymbol{A}^{-H}\boldsymbol{\Omega}_{i}^{E}\boldsymbol{A}^{-1}\boldsymbol{L}^{-1}\boldsymbol{L}\boldsymbol{A}\boldsymbol{A}^{H}\boldsymbol{L}^{H}\boldsymbol{L}^{-H}\boldsymbol{L}^{-1})\\
 & = & \frac{1}{\nu-p}\tr(\boldsymbol{A}\boldsymbol{\Omega}_{i}^{E}\boldsymbol{A}^{H}).
\end{array}
\end{equation}
Therefore we obtain:
\begin{equation}
    \begin{array}{lll}
    T_1(i,j) & = & \delta_{i,j}  \\
    T_2(i,j) & = &  \frac{1}{(\nu-p)^2}\mathbb{E}\left[\tr(\boldsymbol{A}\boldsymbol{\Omega}_{i}^{E}\boldsymbol{A}^{H})\tr(\boldsymbol{A}\boldsymbol{\Omega}_{j}^{E}\boldsymbol{A}^{H})\right] \\
    T_3(i,j) & = & \frac{1}{(\nu-p)^2}\tr(\boldsymbol{\Omega}_{i}^{E})\mathbb{E}\left[\tr(\boldsymbol{A}\boldsymbol{\Omega}_{j}^{E}\boldsymbol{A}^{H})\right]
    \end{array}.
\end{equation}
Finally the new formulation for $\FisherMat^{AI}_{\textup{prior}}$ could be reduced to:
\begin{equation}
\begin{array}{lll}
\left(\FisherMat^{AI}_{\textup{prior}}\right)_{i,j}& = & \left[ (\nu+p)^2 \tr(\boldsymbol{\Omega}_{i}^{E})\tr(\boldsymbol{\Omega}_{j}^{E}) +\mathbb{E}\left[\tr(\boldsymbol{A}\boldsymbol{\Omega}_{i}^{E}\boldsymbol{A}^{H})\tr(\boldsymbol{A}\boldsymbol{\Omega}_{j}^{E}\boldsymbol{A}^{H})\right] \right. \\
&& -(\nu+p)\tr(\boldsymbol{\Omega}_{i}^{E})\mathbb{E}\left[\tr(\boldsymbol{A}\boldsymbol{\Omega}_{j}^{E}\boldsymbol{A}^{H})\right]   \\
& & \left. -(\nu+p)\tr(\boldsymbol{\Omega}_{j}^{E})\mathbb{E}\left[\tr(\boldsymbol{A}\boldsymbol{\Omega}_{i}^{E}\boldsymbol{A}^{H})\right]  \right],
  \end{array}
  \label{eq:fiw_2_nat_comp}
\end{equation}
where the expectation is taken over $\boldsymbol{\Sigma}$. Therefore we have to compute:
\begin{equation}
   \begin{array}{l}
     \left(\boldsymbol{f}^{AI}_{\textup{prior}}\right)_{i}=  \mathbb{E}\left[\tr(\boldsymbol{A}\boldsymbol{\Omega}_{i}^{E}\boldsymbol{A}^{H})\right] \\
     \left(\bar{\FisherMat}^{AI}_{\textup{prior}}\right)_{i,j} = \mathbb{E}\left[\tr(\boldsymbol{A}\boldsymbol{\Omega}_{i}^{E}\boldsymbol{A}^{H})\tr(\boldsymbol{A}\boldsymbol{\Omega}_{j}^{E}\boldsymbol{A}^{H})\right]
\end{array} ,
\label{eq:def_f_F}
\end{equation}
where $\boldsymbol{A}$ is a lower triangular matrix and where all the elements are independent random variables : 
\begin{equation}
\begin{array}{lll}
a_{i,j} & \sim & \mathcal{N_{\mathbb{C}}}(0,1),\,i>j\\
a_{i,i}^{2} & \sim & \frac{1}{2}\chi_{2\left(\nu-i+1\right)}^{2}
\end{array}.
\end{equation}

\subsection{Computation of \eqref{eq:def_f_F}}

% and $\bar{\FisherMat}^{AI}_{\textup{prior}}$ $\FisherMat^{AI}_{\textup{prior}}$

Instead of using the notation $\boldsymbol{\Omega}_{i}^{E}$ for the basis, we come back to $\boldsymbol{\Omega}_{ii}^{E}$, $\boldsymbol{\Omega}_{mn}^{E}$ and $\boldsymbol{\Omega}_{mn}^{h-E}$ which is recalled in the following:
\begin{itemize}
\item $\boldsymbol{\Omega}_{ii}^{E}$ is an $p$ by $p$ symmetric
matrix whose $i$th diagonal element is one, zeros elsewhere 
\item $\boldsymbol{\Omega}_{mn}^{E}$ is an $p$ by $p$ symmetric
matrix whose $mn$th and $nm$th elements are both $2^{-1/2}$, zeros
elsewhere. By convention, we use $m>n$.
\item $\boldsymbol{\Omega}_{mn}^{h-{E}}$ is an $p$ by $p$ Hermitian
matrix whose $mn$th element is $2^{-1/2}\sqrt{-1}$, and $nm$th
element is $-2^{-1/2}\sqrt{-1}$, zeros elsewhere ($m>n$).
\end{itemize}

For the calculation of $\boldsymbol{f}^{AI}_{\textup{prior}}$, we have 
\begin{equation}
\left(\boldsymbol{f}^{AI}_{\textup{prior}}\right)_{i}=\mathbb{E}[\tr(\boldsymbol{A}\boldsymbol{\Omega}_{i}^{E}\boldsymbol{A}^{H})]=\mathbb{E}[\tr(\boldsymbol{\Omega}_{i}^{E}\boldsymbol{A}^{H}\boldsymbol{A})]=\tr(\boldsymbol{\Omega}_{i}^{E}\mathbb{E}[\boldsymbol{A}^{H}\boldsymbol{A}]).
\end{equation}
First, we are looking for the $\left(i,j\right)-$element of the matrix $\boldsymbol{A}^{H}\boldsymbol{A}$. The $i^{th}$ row of the matrix $\boldsymbol{A}^{T}$ is the following vector 
\begin{equation}
\left.\boldsymbol{A}^{H}\right|_{i^{th}\,row}=\left(\underset{i-1\:elements}{\underbrace{\begin{array}{ccc}
0 & \cdots & 0\end{array}}}\underset{p-i+1\:elements}{\underbrace{\begin{array}{cccc}
a_{i,i}^{*} & a_{i+1,i}^{*} & \cdots & a_{p,i}^{*}\end{array}}}\right),
\end{equation}
and the $j^{th}$ column of the matrix $\boldsymbol{A}$ is the following vector 
\begin{equation}
\left.\boldsymbol{A}\right|_{j^{th}\,column}=\left(\underset{j-1\:elements}{\underbrace{\begin{array}{ccc}
0 & \cdots & 0\end{array}}}\underset{p-j+1\:elements}{\underbrace{\begin{array}{cccc}
a_{j,j} & a_{j+1,j} & \cdots & a_{p,j}\end{array}}}\right)^{T}.
\end{equation}

Then, if $i>j$ the $\left(i,j\right)-$element of the matrix $\boldsymbol{A}^{H}\boldsymbol{A}$
is given by 
\begin{equation}
\left.\boldsymbol{A}^{H}\boldsymbol{A}\right|_{i,j}=\left.\boldsymbol{A}^{H}\boldsymbol{A}\right|_{j,i}=a_{i,i}^{*}a_{i,j}+a_{i+1,i}^{*}a_{i+1,j}+\cdots+a_{p,i}^{*}a_{p,j},
\end{equation}
and $\mathbb{E}[\left.\boldsymbol{A}^{H}\boldsymbol{A}\right|_{i,j}]=0$ since all the random variables are independent and $\mathbb{E}[a_{i,j}]=0$ for $i\neq j$. Consequently, with the same reasoning when $j>i$ one can conclude that $\mathbb{E}[\boldsymbol{A}^{H}\boldsymbol{A}]$ is a diagonal matrix. If now $i=j$, then 
\begin{equation}
\left.\boldsymbol{A}^{H}\boldsymbol{A}\right|_{i,i}=\underset{p-i+1\,elements}{\underbrace{a_{i,i}^{2}+\left| a_{i+1,i}\right|^{2}+\cdots+\left| a_{p,i}\right| ^{2}}}.
\end{equation}
Since $a_{i,i}^{2}\sim\frac{1}{2}\chi_{2\left(\nu-i+1\right)}^{2}$ and $a_{i,j\neq i}\sim\mathcal{CN}\left(0,1\right)$, and the independence of all the random variables, one has $\mathbb{E}\left[a_{i,i}^{2}\right]=\nu-i+1$ (which appears once) and $\mathbb{E}\left[\left| a_{i,j\neq i}\right| ^{2}\right]=1$ (which appears $p-i$ times) and 
\begin{equation}
\mathbb{E}\left[\left.\boldsymbol{A}^{H}\boldsymbol{A}\right|_{i,i}\right]=\nu-i+1+p-i=\nu+p-2i+1.
\end{equation}
Now if $\boldsymbol{\Omega}_{i}^{E}:=\boldsymbol{\Omega}_{ii}^{E}$, one obtains $\left(\boldsymbol{f}^{AI}_{\textup{prior}}\right)_{i}=\tr(\boldsymbol{\Omega}_{ii}^{E}\mathbb{E}[\boldsymbol{A}^{H}\boldsymbol{A}])=\mathbb{E}\left[\left.\boldsymbol{A}^{H}\boldsymbol{A}\right|_{i,i}\right]=\nu+p-2i+1$
and, if $\boldsymbol{\Omega}_{i}^{E}:=\boldsymbol{\Omega}_{mn}^{E}$ (or equivalently $\boldsymbol{\Omega}_{mn}^{h-E}$),
then $\boldsymbol{\Omega}_{mn}^{E}\mathbb{E}[\boldsymbol{A}^{H}\boldsymbol{A}]$
is a matrix with zeros on its diagonal and its trace is then equal to zero. Consequently, 
\begin{equation}
\left(\boldsymbol{f}^{AI}_{\textup{prior}}\right)_{i}=\left\{ \begin{array}{c}
0\:if\:\boldsymbol{\Omega}_{i}^{E}:=\boldsymbol{\Omega}_{mn}^{E}\:or\:\boldsymbol{\Omega}_{mn}^{h-E} \\
\nu+p-2i+1\:if\,\boldsymbol{\Omega}_{i}^{E}:=\boldsymbol{\Omega}_{ii}^{E} 
\end{array}\right. .
\end{equation}

Now we are interested in the calculation of $\bar{\FisherMat}^{AI}_{\textup{prior}}$. Let's first analyze the term $\tr(\boldsymbol{\Omega}_{i}^{E}\boldsymbol{A}^{H}\boldsymbol{A})$ (since we know from the calculus of $\FisherMat^{AI}_{\textup{prior}}$ the structure of $\left.\boldsymbol{A}^{H}\boldsymbol{A}\right|_{i,j}$). 

If $\boldsymbol{\Omega}_{i}^{E}:=\boldsymbol{\Omega}_{ii}^{E}$ then 
\begin{equation}
\tr(\boldsymbol{\Omega}_{ii}^{E}\boldsymbol{A}^{H}\boldsymbol{A})=\left.\boldsymbol{A}^{H}\boldsymbol{A}\right|_{i,i}=\underset{p-i+1\,elements}{\underbrace{a_{i,i}^{2}+\left| a_{i+1,i}\right| ^{2}+\cdots+\left| a_{p,i}\right| ^{2}}}.
\end{equation}

If $\boldsymbol{\Omega}_{i}^{E}:=\boldsymbol{\Omega}_{mn}^{E}$ then 
\begin{equation}
\begin{array}{lll}
\tr(\boldsymbol{\Omega}_{mn}^{E}\boldsymbol{A}^{H}\boldsymbol{A})&=&\frac{1}{\sqrt{2}}\left(a_{m,m}a_{m,n}+a_{m+1,m}^{*}a_{m+1,n}+\cdots \right. \\
& & \left. +a_{p,m}^{*}a_{p,n}+a_{m,m}a_{m,n}^{*}+a_{m+1,m}a_{m+1,n}^{*}+\cdots+a_{p,m}a_{p,n}^{*}\right),
\end{array}
\end{equation}
where we have $2\left(p-i+1\right)$ elements. 

If $\boldsymbol{\Omega}_{i}^{E}:=\boldsymbol{\Omega}_{mn}^{h-{E}}$ then
\begin{equation}
\begin{array}{lll}
\tr(\boldsymbol{\Omega}_{mn}^{h-{E}}\boldsymbol{A}^{H}\boldsymbol{A})&=&\frac{\sqrt{-1}}{\sqrt{2}}\left(-a_{m,m}a_{m,n}-a_{m+1,m}^{*}a_{m+1,n}-\cdots \right. \\
& & \left. -a_{p,m}^{*}a_{p,n}+a_{m,m}a_{m,n}^{*}+a_{m+1,n}a_{m+1,n}^{*}+\cdots+a_{p,m}a_{p,n}^{*}\right).
\end{array}
\end{equation}

Second, let's study $\tr(\boldsymbol{\Omega}_{i}^{E}\boldsymbol{A}^{H}\boldsymbol{A})\tr(\boldsymbol{\Omega}_{j}^{E}\boldsymbol{A}^{H}\boldsymbol{A})$ and its expectation. Several cases appear:

\paragraph*{Case 1} If $\boldsymbol{\Omega}_{i}^{E}:=\boldsymbol{\Omega}_{j}^{E}:=\boldsymbol{\Omega}_{ii}^{E}$
\begin{equation}
\begin{array}{lll}
\left(\bar{\FisherMat}^{AI}_{\textup{prior}}\right)_{i,j} & = & \mathbb{E}[\tr^{2}(\boldsymbol{\Omega}_{ii}^{E}\boldsymbol{A}^{H}\boldsymbol{A})]\\
 & = & \mathbb{E}\left[\left(a_{i,i}^{2}+\left| a_{i+1,i}\right| ^{2}+\cdots+\left| a_{p,i}\right| ^{2}\right)^{2}\right]\\
 & = & \text{VAR}\left(a_{i,i}^{2}+\left| a_{i+1,i}\right| ^{2}+\cdots+\left| a_{p,i}\right| ^{2}\right)+\left(\mathbb{E}\left[a_{i,i}^{2}+\left| a_{i+1,i}\right| ^{2}+\cdots+\left| a_{p,i}\right| ^{2}\right]\right)^2\\
 & = & \left(\nu-i+1\right)+\left(\nu+p-2i+1\right)^{2}.
\end{array}
\end{equation}

\paragraph*{Case 2} If $\boldsymbol{\Omega}_{i}^{E}:=\boldsymbol{\Omega}_{ii}^{E}$
and $\boldsymbol{\Omega}_{j}^{E}:=\boldsymbol{\Omega}_{jj}^{E}$ with $i\neq j$ 
\begin{equation}
\begin{array}{lll}
\left(\bar{\FisherMat}^{AI}_{\textup{prior}}\right)_{i,j} & = & \mathbb{E}[\tr(\boldsymbol{\Omega}_{ii}^{E}\boldsymbol{A}^{H}\boldsymbol{A})\tr(\boldsymbol{\Omega}_{jj}^{E}\boldsymbol{A}^{H}\boldsymbol{A})]\\
 & = & \mathbb{E}\left[\left(\underset{p-i+1\,elements}{\underbrace{a_{i,i}^{2}+\left| a_{i+1,i}\right| ^{2}+\cdots+\left| a_{p,i}\right| ^{2}}}\right)\left(\underset{p-j+1\,elements}{\underbrace{a_{j,j}^{2}+\left| a_{j+1,j}\right| ^{2}+\cdots+\left| a_{p,j}\right| ^{2}}}\right)\right]\\
 & = & \left(\nu-2i+p+1\right)\left(\nu-2j+p+1\right).\\
\\
\end{array}
\end{equation}

\paragraph*{Case 3} If $\boldsymbol{\Omega}_{i}^{E}:=\boldsymbol{\Omega}_{ii}^{E}$
and $\boldsymbol{\Omega}_{j}^{E}:=\boldsymbol{\Omega}_{mn}^{E}$
(with $m>n$ but where possibly $m$ or $n$ can be equal to $i$)
\begin{equation}
\begin{array}{l}
\left(\bar{\FisherMat}^{AI}_{\textup{prior}}\right)_{i,j}=\frac{1}{\sqrt{2}}\mathbb{E}\left[\left(a_{i,i}^{2}+\left| a_{i+1,i}\right| ^{2}+\cdots+\left| a_{p,i}\right| ^{2}\right) \right. \\ \left. \left( a_{m,m}a_{m,n}+a_{m+1,m}^{*}a_{m+1,n}+\cdots+a_{p,m}^{*}a_{p,n}+a_{m,m}a_{m,n}^{*}+a_{m+1,m}a_{m+1,n}^{*}+\cdots+a_{p,m}a_{p,n}^{*}\right)\right].
\end{array}
\end{equation}
Each terms is the product of an expectation of up to 3 random variables. More precisely, $\left(\bar{\FisherMat}^{AI}_{\textup{prior}}\right)_{i,j}$ is made with terms that can only be in one of following situations: (i) $\mathbb{E}\left[a_{i,i}^{2}a_{m,m}a_{m,n}\right]=\mathbb{E}\left[a_{i,i}^{2}a_{m,m}\right]\mathbb{E}\left[a_{m,n}\right]=0$ since $m\neq n$ and $\mathbb{E}\left[a_{m,n}\right]=0$ (same thing when we have $\mathbb{E}\left[a_{i,i}^{2}a_{m,m}a_{m,n}^{*}\right]$). (ii) with $N=\left\{ 1,\ldots,p-m\right] $, $\mathbb{E}\left[a_{i,i}^{2}a_{m+N,m}^{*}a_{m+N,n}\right]=\mathbb{E}\left[a_{i,i}^{2}\right]\mathbb{E}\left[a_{m+N,m}^{*}\right]\mathbb{E}\left[a_{m+N,n}\right]=0$ since $m\neq n$ and $N>0$ one always has $a_{m+N,m}\neq a_{m+N,n}\neq a_{i,i}$
(same thing when we have $\mathbb{E}\left[a_{i,i}^{2}a_{m+N,m}a_{m+N,n}^{*}\right]$). (iii) with $Q=\left\{ 1,\ldots,p-m\right] $, $\mathbb{E}\left[a_{i+Q,i}^{2}a_{m,m}a_{m,n}\right]=0$ because, the random variables $a_{i+Q,i}^{2}$ and $a_{m,n}$ are the same (i.e. not independent) if and only if $m=i+Q$ and $n=i$ which leads to $\mathbb{E}\left[a_{i+Q,i}^{2}a_{m,m}a_{m,n}\right]=\mathbb{E}\left[a_{m,m}\right]\mathbb{E}\left[a_{m,n}^{3}\right]=0$ since $m\neq n$ and $\mathbb{E}\left[a_{m,n}^{3}\right]=0$ (recall that $a_{m,n}\sim\mathcal{CN}\left(0,1\right)$). Otherwise, $a_{i+Q,i}^{2}$
and $a_{m,n}$ are independent, then $\mathbb{E}\left[a_{i+Q,i}^{2}a_{m,m}a_{m,n}\right]=\mathbb{E}\left[a_{m,m}\right]\mathbb{E}\left[a_{i+Q,i}^{2}\right]\mathbb{E}\left[a_{m,n}\right]=0$ since $m\neq n$ and $\mathbb{E}\left[a_{m,n}\right]=0$. (iv) with $N=\left\{ 1,\ldots,p-m\right] $ and $Q=\left\{ 1,\ldots,p-m\right] $, $\mathbb{E}\left[\left| a_{i+Q,i}\right| ^{2}a_{m+N,m}^{*}a_{m+N,n}\right]=0$ because $a_{m+N,m}\bot a_{m+N,n}$ since $m>n$ and $N>0$ (same thing when we have $\mathbb{E}\left[\left| a_{i+Q,i}\right| ^{2}a_{m+N,m}a_{m+N,n}^{*}\right]$). Consequently, for the Case 3, one always has 
\begin{equation}
\left(\bar{\FisherMat}^{AI}_{\textup{prior}}\right)_{i,j}=0,
\end{equation}
and, trivially, the case $\boldsymbol{\Omega}_{j}^{E}:=\boldsymbol{\Omega}_{jj}^{E}$ and $\boldsymbol{\Omega}_{i}^{E}:=\boldsymbol{\Omega}_{mn}^{E}$ (with $>n$ but where possibly $m$ or $n$ can be equal to $j$) leads also to $\left(\bar{\FisherMat}^{AI}_{\textup{prior}}\right)_{i,j}=0$.

\paragraph*{Case 4} If $\boldsymbol{\Omega}_{i}^{E}:=\boldsymbol{\Omega}_{mn}^{E}$
(with $m>n$) and $\boldsymbol{\Omega}_{j}^{E}:=\boldsymbol{\Omega}_{kl}^{E}$
(with $k>l$) 
\begin{equation}
\begin{array}{lll}
\left(\bar{\FisherMat}^{AI}_{\textup{prior}}\right)_{i,j}&=&\frac{1}{2}\mathbb{E}\left[\left(a_{m,m}a_{m,n}+a_{m+1,n}^{*}a_{m+1,n}+\cdots \right. \right. \\
& & \left. \left. +a_{p,m}^{*}a_{p,n}+a_{m,m}a_{m,n}^{*} 
+a_{m+1,m}a_{m+1,n}^{*}+\cdots+a_{p,m}a_{p,n}^{*}\right) \right.\\ & & \left.\left( a_{k,k}a_{k,l}+a_{k+1,k}^{*}a_{k+1,l}+\cdots \right. \right. \\
& & \left. \left. +a_{p,k}^{*}a_{p,l}+a_{k,k}a_{k,l}^{*}+a_{k+1,k}a_{k+1,l}^{*}+\cdots+a_{p,k}a_{p,l}^{*}\right)\right].
\end{array}
\end{equation}

Concerning terms such that $\mathbb{E}[a_{m,m}a_{m,n}a_{k,k}a_{k,l}^{*}]$
(or $\mathbb{E}[a_{m,m}a_{m,n}^{*}a_{k,k}a_{k,l}]$), if $m=k$ and
$n=l$, one has $\mathbb{E}[a_{m,m}a_{m,n}a_{k,k}a_{k,l}^{*}]=\mathbb{E}[a_{m,m}^{2}]\mathbb{E}[\left| a_{m,n}\right| ^{2}]=\nu-m+1$
since $a_{m,m}^{2}\sim\chi_{\nu-m+1}^{2}$ and $\mathbb{E}[\left| a_{m,n}\right| ^{2}]=1$.
Again if $m=k$ and $n=l$ one also have terms such that $\mathbb{E}[a_{m,m}a_{m,n}^{*}a_{k,k}a_{k,l}^{*}]=\mathbb{E}[a_{m,m}a_{m,n}a_{k,k}a_{k,l}]=0$
since $\mathbb{E}[a_{m,n}^{2}]=0$. Otherwise, $\mathbb{E}[a_{m,m}a_{m,n}a_{k,k}a_{k,l}^{*}]=0$
and $\mathbb{E}[a_{m,m}a_{m,n}^{*}a_{k,k}a_{k,l}]=0$ because $a_{m,n}\bot a_{k,l}$
and $\mathbb{E}\left[a_{k,l}\right]=\mathbb{E}\left[a_{m,n}\right]=0$.
Concerning terms such that $\mathbb{E}[a_{m,m}a_{m,n}a_{k+N,k}^{*}a_{k+N,l}]$,
$\mathbb{E}[a_{m,m}a_{m,n}a_{k+N,k}a_{k+N,l}^{*}]$, $\mathbb{E}[a_{m+Q,m}^{*}a_{m+Q,n}a_{k,k}a_{k,l}]$
and $\mathbb{E}[a_{m+Q,m}a_{m+Q,n}^{*}a_{k,k}a_{k,l}]$, they are
always equal to $0$ since $a_{k+N,k}\bot a_{k+N,l}$ and $a_{m+Q,m}\bot a_{m+Q,n}$.

Finally, $\mathbb{E}[a_{m+Q,m}^{*}a_{m+Q,n}a_{k+N,k}a_{k+N,l}^{*}]=\mathbb{E}[a_{m+Q,m}a_{m+Q,n}^{*}a_{k+N,k}^{*}a_{k+N,l}]=\mathbb{E}[\left| a_{m+Q,m}\right| ^{2}\left| a_{m+Q,n}\right| ^{2}]=1$ if and only if $m=k$, $n=l$ and $P=N$ (otherwise one always have at least on zero mean random variable independent other the three others, or the fact that $\mathbb{E}[a_{m,n}^{2}]=0$, and these terms are equal to $0$). Consequently for case 4, if $\boldsymbol{\Omega}_{i}^{E}:=\boldsymbol{\Omega}_{mn}^{E}=\boldsymbol{\Omega}_{j}^{E}$
\begin{equation}
\left(\bar{\FisherMat}^{AI}_{\textup{prior}}\right)_{i,j}=\left(\nu-m+1+p-m\right)=\left(\nu-2m+p+1\right),
\end{equation}
and $\left(\bar{\FisherMat}^{AI}_{\textup{prior}}\right)_{i,j}=0$ otherwise.

\paragraph*{Case 5} If $\boldsymbol{\Omega}_{i}^{E}:=\boldsymbol{\Omega}_{ii}^{E}$
and $\boldsymbol{\Omega}_{j}^{E}:=\boldsymbol{\Omega}_{mn}^{h-{E}}$
(with $m>n$ but where possibly $m$ or $n$ can be equal to $i$)
\begin{equation}
\begin{array}{cc}
\left(\bar{\FisherMat}^{AI}_{\textup{prior}}\right)_{i,j}=\frac{\sqrt{-1}}{\sqrt{2}}\mathbb{E}\left[\left(a_{i,i}^{2}+\left| a_{i+1,i}\right| ^{2}+\cdots+\left| a_{p,i}\right| ^{2}\right)\right. \\ \left. \left( -a_{m,m}a_{m,n}-a_{m+1,m}^{*}a_{m+1,n}-\cdots-a_{p,m}^{*}a_{p,n}+a_{m,m}a_{m,n}^{*}+a_{m+1,n}a_{m+1,n}^{*}+\cdots+a_{p,m}a_{p,n}^{*}\right)\right].
\end{array}
\end{equation}
An analysis, similar to the one provided in case 3, and using the fact that $\forall N=\left\{ 1,\ldots,p-m\right] $, $\mathbb{E}\left[a_{m+N,m}^{*}a_{m+N,n}a_{m+N,m}^{*}\right]=\mathbb{E}\left[a_{m+N,m}^{*}a_{m+N,n}a_{m+N,n}\right]=0$ leads to
\begin{equation}
\left(\bar{\FisherMat}^{AI}_{\textup{prior}}\right)_{i,j}=0.
\end{equation}
Of course, the case $\boldsymbol{\Omega}_{j}^{E}:=\boldsymbol{\Omega}_{jj}^{E}$
and $\boldsymbol{\Omega}_{i}^{E}:=\boldsymbol{\Omega}_{mn}^{h-{E}}$
(with $m>n$ but where possibly $m$ or $n$ can be equal to $j$) leads also to $\left(\bar{\FisherMat}^{AI}_{\textup{prior}}\right)_{i,j}=0$.

\paragraph*{Case 6} If $\boldsymbol{\Omega}_{i}^{E}:=\boldsymbol{\Omega}_{mn}^{E}$
(with $m>n$) and $\boldsymbol{\Omega}_{j}^{E}:=\boldsymbol{\Omega}_{kl}^{h-{E}}$
(with $k>l$)
\begin{equation}
\begin{array}{lll}
\left(\bar{\FisherMat}^{AI}_{\textup{prior}}\right)_{i,j}&=&\frac{\sqrt{-1}}{2}\mathbb{E}\left[\left(a_{m,m}a_{m,n}+a_{m+1,n}^{*}a_{m+1,n}+\cdots \right. \right. \\
& & \left. \left. +a_{p,m}^{*}a_{p,n}+a_{m,m}a_{m,n}^{*}+a_{m+1,m}a_{m+1,n}^{*}+\cdots+a_{p,m}a_{p,n}^{*}\right) \right. \\ 
& & \left. \left( -a_{k,k}a_{k,l}-a_{k+1,k}^{*}a_{k+1,l}-\cdots \right. \right. \\
& & \left. \left. -a_{p,k}^{*}a_{p,l}+a_{k,k}a_{k,l}^{*}+a_{k+1,k}a_{k+1,l}^{*}+\cdots+a_{p,k}a_{p,l}^{*}\right)\right].
\end{array}
\end{equation}
Using all the aforementioned arguments, one find after calculus that
\begin{equation}
\left(\bar{\FisherMat}^{AI}_{\textup{prior}}\right)_{i,j}=0.
\end{equation}
Of course, the case $\boldsymbol{\Omega}_{i}^{E}:=\boldsymbol{\Omega}_{mn}^{h-{E}}$
(with $m>n$) and $\boldsymbol{\Omega}_{j}^{E}:=\boldsymbol{\Omega}_{kl}^{E}$
(with $k>l$) leads also to $\left(\bar{\FisherMat}^{AI}_{\textup{prior}}\right)_{i,j}=0$.

\paragraph*{Case 7} If $\boldsymbol{\Omega}_{i}^{E}:=\boldsymbol{\Omega}_{mn}^{h-{E}}$ (with $m>n$) and $\boldsymbol{\Omega}_{j}^{E}:=\boldsymbol{\Omega}_{kl}^{h-{E}}$ (with $k>l$)
\begin{equation}
\begin{array}{lll}
\left(\bar{\FisherMat}^{AI}_{\textup{prior}}\right)_{i,j}&=&-\frac{1}{2}\mathbb{E}\left[\left(-a_{m,m}a_{m,n}-a_{m+1,m}^{*}a_{m+1,n}-\cdots \right. \right. \\ & & \left. \left. -a_{p,m}^{*}a_{p,n}+a_{m,m}a_{m,n}^{*}+a_{m+1,m}a_{m+1,n}^{*}+\cdots+a_{p,m}a_{p,n}^{*}\right) \right. \\ & &  \left. \left( -a_{k,k}a_{k,l}-a_{k+1,k}^{*}a_{k+1,l}-\cdots \right. \right. \\ & & \left. \left. -a_{p,k}^{*}a_{p,l}+a_{k,k}a_{k,l}^{*}+a_{k+1,k}a_{k+1,l}^{*}+\cdots+a_{p,k}a_{p,l}^{*}\right)\right].
\end{array}
\end{equation}
An analysis, similar to the one provided in case 4, shows that $\left(\bar{\FisherMat}^{AI}_{\textup{prior}}\right)_{i,j}\neq 0$ if and only if $m=k$ and $n=l$. In this case, one obtains easily
\begin{equation}
\left(\bar{\FisherMat}^{AI}_{\textup{prior}}\right)_{i,j}=\nu-2m+p+1.
\end{equation}

\paragraph*{Final result} Then one can summary the value of $\left(\bar{\FisherMat}^{AI}_{\textup{prior}}\right)_{i,j}\neq 0$ as:
\begin{itemize}
\item If $\boldsymbol{\Omega}_{i}^{E}:=\boldsymbol{\Omega}_{ii}^{E}:=\boldsymbol{\Omega}_{j}^{E}$
then $\left(\bar{\FisherMat}^{AI}_{\textup{prior}}\right)_{i,j}=\left(\nu-i+1\right)+\left(\nu+p-2i+1\right)^{2}$.
\item If $\boldsymbol{\Omega}_{i}^{E}:=\boldsymbol{\Omega}_{ii}^{E}$ and $\boldsymbol{\Omega}_{j}^{E}:=\boldsymbol{\Omega}_{jj}^{E}$
with $i\neq j$ then $\left(\bar{\FisherMat}^{AI}_{\textup{prior}}\right)_{i,j}=\left(\nu-2i+p+1\right)\left(\nu-2j+p+1\right)$.
\item If $\boldsymbol{\Omega}_{i}^{E}:=\boldsymbol{\Omega}_{mn}^{E}=\boldsymbol{\Omega}_{j}^{E}$
then $\left(\bar{\FisherMat}^{AI}_{\textup{prior}}\right)_{i,j}=\nu-2m+p+1$.
\item If $\boldsymbol{\Omega}_{i}^{E}:=\boldsymbol{\Omega}_{mn}^{h-E}=\boldsymbol{\Omega}_{j}^{E}$
then $\left(\bar{\FisherMat}^{AI}_{\textup{prior}}\right)_{i,j}=\nu-2m+p+1$.
\end{itemize}

\end{document}